\documentclass{amsart}
\usepackage{latexsym,amsxtra,amscd,ifthen}
\usepackage{amsfonts}
\usepackage{verbatim}
\usepackage{amsmath}
\usepackage{amsthm}
\usepackage{amssymb}

\allowdisplaybreaks[4]

\usepackage{array} 
\usepackage{amscd}
\usepackage{graphicx} 
\usepackage[mathscr]{eucal}
\usepackage[usenames,dvipsnames]{color}
\usepackage{bbm}
\usepackage[all, 2cell]{xy}
\UseAllTwocells

\usepackage[colorlinks]{hyperref}

\newtheorem{theorem}[equation]{Theorem}
\newtheorem{lemma}[equation]{Lemma}
\newtheorem{proposition}[equation]{Proposition}
\newtheorem{corollary}[equation]{Corollary}

\theoremstyle{remark}
\newtheorem{definition}[equation]{Definition}

\newenvironment{example}[1][Example]{\begin{trivlist}
\item[\hskip \labelsep {\bfseries #1}]}{\end{trivlist}}

\numberwithin{equation}{section}

\newcommand{\nc}{\newcommand}
\nc{\dmo}{\DeclareMathOperator}

\dmo{\Arr}{Arr}
\dmo{\Ch}{Ch}
\dmo{\Der}{D}
\dmo{\Ho}{Ho}
\dmo{\Mod}{Mod}
\dmo{\op}{op}
\dmo{\QCoh}{Qcoh}
\dmo{\Spec}{Spec}

\nc{\aka}{{a.\,k.\,a.\ }}
\nc{\eg}{\textsl{e.g.}}
\nc{\ie}{\textsl{i.e.\ }}

\nc{\cat}[1]{\mathscr #1}
\nc{\Cat}{\mathsf{Cat}}
\nc{\CAT}{\mathsf{CAT}}
\nc{\DER}{\bbD\mathrm{er}}
\nc{\inv}{^{-1}}
\nc{\isotoo}{\,\otoo{\cong}\,}
\nc{\HO}{\mathbb{HO}}
\nc{\Loop}{\underline{\bbN}}
\nc{\MMod}{\text{-}\kern-0.1em\Mod}%
\nc{\otoo}[1]{\overset{#1}{\,\too\,}}
\nc{\John}[1]{{[\color{Green}#1]}}
\nc{\Paul}[1]{{[\color{Blue}#1]}}
\nc{\too}{\mathop{\longrightarrow}\limits}
\nc{\threestars}{\medbreak\begin{center}*\ *\ *\end{center}}

\newcommand{\swtrans}
{\mathbin{\rotatebox[origin=c]{225}{$\Rightarrow$}}}
\newcommand{\netrans}
{\mathbin{\rotatebox[origin=c]{45}{$\Rightarrow$}}}

\newcommand{\nwtrans}
{\mathbin{\rotatebox[origin=c]{135}{$\Rightarrow$}}}

\nc{\bbA}{\mathbb{A}}
\nc{\bbD}{\mathbb{D}}
\nc{\bbE}{\mathbb{E}}
\nc{\bbN}{\mathbb{N}}
\nc{\bbZ}{\mathbb{Z}}

\title{Affine lines over derivators: Properties}
\date{\today}
\author{John Zhang}
\address{John Zhang, Mathematics Department, UCLA, Los Angeles, CA 90095-1555, USA}
\email{jmzhang@math.ucla.edu}
\urladdr{http://www.math.ucla.edu/$\sim$jmzhang}

\begin{document}

\maketitle

\begin{abstract}
Utilizing the construction of $\bbA^1_{\bbD}$ of \cite{BalmerZhang16}, we prove a number of properties between $\bbD$ and $\mathbb{A}^1_{\bbD}$, the most important of which is a universal property that mirrors the universal property of $R[t]$ over $R$ for rings. The universal property generalizes easily to affine space $\bbA^n_{\bbD}$ over the derivator $\bbD$.
\end{abstract}
\tableofcontents

\section{Introduction}

We begin with the following motivating question: if we have a morphism of schemes $V\rightarrow X$ and we know the triangulated structure of the derived category of $X$, can the derived category of $V$ be recovered as a triangulated category? The easiest case is the inclusion of an open subscheme $U\hookrightarrow X$, which is described as a Bousfield localization $D(U)\cong D(X)/D_Z(X)$, where $D_Z(X)$ is the full subcategory of complexes with cohomology supported on $Z=X-U$, see \cite{ThomasonTrobaugh90}. In the more general case of a separated, \'etale morphism Balmer has the following interesting result.

\begin{theorem}\cite[Theorem 3.5]{Balmer14pp}
Let $f:V\rightarrow X$ be a separated, \'etale morphism between quasi-compact and quasi-separated schemes. There is a monad $\bbA_f$, induced by the adjunction $f^*\dashv Rf_*$ on $D(X)$, with endofunctor $\bbA_f=Rf_*\circ f^*$, unit $\eta\colon\textrm{Id}_{D(X)}\rightarrow Rf_*\circ f^*$ induced by the unit for the adjunction, and multiplication $\mu\colon\bbA_f\circ \bbA_f\rightarrow\bbA_f$ induced by the counit $\epsilon\colon f^*\circ Rf_*\rightarrow\textrm{Id}_{D(X)}$. 

Moreover, $\bbA_f$ is a separable, exact monad, and there is a unique (triangulated) equivalence of triangulated categories $$E\colon D(V)\cong (\bbA_f\MMod)_{D(X)}$$ between the derived category of $V$ and the category of $\bbA_f$-modules in the derived category of $X$. Moreover, $f^*\colon D(X)\rightarrow D(V)$ becomes isomorphic to the extension of scalars $F_{\bbA_f}$ along $\bbA_f$, while $Rf_*\colon D(V)\rightarrow D(X)$ becomes isomorphic to the forgetful functor $U_{\bbA_f}$.
\end{theorem}

Naturally, we direct the readers to \cite{Balmer14pp} for a full discussion of this result. For us, the interest is that we can construct the derived category of $V$ from $X$. However, both have the limitation that the dimensions of either the open subset $U$ or the \'etale cover $V$ must be the same as that of $X$. As a result, we can only investigate schemes with the same dimension as $X$. In \cite{BalmerZhang16}, we reconstruct the derived category of $\bbA^1_X$ and more generally the derived categories of $\bbA^n_X$ via the formalism of \emph{derivators} as values of the derivator associated to the scheme $X$. In contrast to previous work, we are not beholden to the use of triangulated categories. With derivator formalism, in \cite{BalmerZhang16} we define the notion of $\bbA^1$ and generally $\bbA^n$ of a derivator and here our main object of interest is the $\bbA^1$-derivator. 

Our line of inquiry is motivated by a number of interactions between the categories $R\MMod$ and $R[t]\MMod$, where $R$ is a commutative ring. In Section 2, we give the main definitions regarding derivators, as well as a number of necessary results. Section 3 recalls the construction of $\bbA^1$ over a derivator and state three important morphisms between a derivator $\bbD$ and its affine line $\bbA^1_{\bbD}$. 

In Section 4, we bring in the notion of \emph{monoidal derivator}, and we describe $\bbA^1_{\bbD}$ as a monoidal derivator if $\bbD$ is a monoidal derivator. In Section 5, we check that the morphisms defined in Section 3 are strong monoidal and also define a class of \emph{evaluation at $\alpha$} morphisms. Section 6 defines the notion of a closed monoidal derivator, and we prove that if $\bbD$ is a closed monoidal derivator, then $\bbA^1_{\bbD}$ with the monoidal structure of Section 4 is also a closed monoidal derivator. Section 7 contains two results, Theorem \eqref{maintheorem1} and Theorem \eqref{maintheorem2} that together describe the \emph{universal property} that $\bbA^1_{\bbD}$ has over $\bbD$. 

For a ring $R$, one way to think of an $R[t]$-module is as an $R$-module with a distinguished $R$-module endomorphism, corresponding to the multiplication by $t$. If we let $\underline{\bbN}$ denote the category with a single object $\bullet$ and $\textrm{Hom}(\bullet,\bullet)$ be the monoid $\bbN$ with composition given by addition, it is easy to see that $R[t]\MMod\cong R\MMod^{\underline{\bbN}}$. This extends to an isomorphism on the categories of unbounded chain complexes, see \cite[Lemma 9]{BalmerZhang16} and to the associated derivators, \cite[Theorem 5]{BalmerZhang16}. 

 We first pick out three distinguished extension of scalars functors between $R\MMod$ and $R[t]\MMod$:

\begin{enumerate}
\item The functor $R\MMod\rightarrow R[t]\MMod$ given by $M\mapsto M\otimes_R R[t]$. We call this the \emph{structure morphism}, and it is the extension of scalars functor along $R\hookrightarrow R[t]$. 
\item The functor $R[t]\MMod\rightarrow R\MMod$ given by $X\mapsto X\otimes_{R[t]} R[t]/(t-1)$. We call this the \emph{evaluation at 1 morphism}, and it is the extension of scalars functor along $R[t]\rightarrow R[t]/(t-1)\cong R$.
\item The functor $R[t]\MMod\rightarrow R\MMod$ given by $X\mapsto X\otimes_{R[t]} R[t]/(t)$. We call this the \emph{evaluation at 0 morphism}, and it is the extension of scalars functor along $R[t]\rightarrow R[t]/(t)\cong R$.
\end{enumerate}

There are many other functors $R[t]\MMod\rightarrow R\MMod$, for example the extension of scalars $R[t]\rightarrow R[t]/(t-r)\cong R$, for any $r\in R$, or a restriction of scalars along those morphisms. However, we give our preference to monoidal functors.

In the categories $R[t]\MMod$ and $R\MMod$, each has the structure of a symmetric monoidal category via the tensor product of modules. Viewing $R[t]\MMod$ as $R\MMod^{\underline{\bbN}}$, we'd like to describe the tensor product $-\otimes_{R[t]}-$ in terms of $-\otimes_R -$. So suppose we have two $R[t]$-modules, which we write as $(M,t_M)$ and $(N,t_N)$, here $t_M$ and $t_N$ being the (distinguished) multiplication by $t$  morphisms on the underlying $R$-modules $M$ and $N$. Merely considering the module $M\otimes_R N$, there are two distinguished maps on it, $t_M\otimes 1$ or $1\otimes t_N$. Naturally, for $M\otimes_{R[t]} N$ we want those two maps to be equal, so we take the coequalizer of $t_M\otimes 1$ and $1\otimes t_N$ to get the underlying $R$-module, and now $t_M\otimes 1$ and $1\otimes t_N$ on the coequalizer are now the same morphism.

Another equally good way to think about this situation is to consider $\underline{\bbN}$ as a symmetric monoidal category, with the tensor product of maps given by addition on $\bbN$. Then there is a natural symmetric monoidal structure defined on $R\MMod^{\underline{\bbN}}$ via the Day convolution. In fact, this is a good model for how we construct the monoidal derivator structure on $\bbA^1_{\bbD}$. It also admits an easier generalization to shifted derivators other than $\bbA^1_{\bbD}$. 

The next step is to define \emph{evaluation at $\alpha$} morphisms. In our context, for a ring $R$ and an element $r\in R$, we can define a functor $$-\otimes_{R[t]} R[t]/(t-r)\colon R[t]\MMod\rightarrow R\MMod,$$ by thinking of $R[t]/(t-r)\cong R$. These functors are monoidal, and unsurprisingly are sections of the extension of scalars along $R\hookrightarrow R[t]$. Our aim is to reproduce such functors in the general context of derivators. The trick is to look at the $R[t]$-module $R[t]/(t-r)$, and think of it as having underlying module $R$, multiplication by $t$ endomorphism $r$, and then think of $R$ as the monoidal unit in $R\MMod$, so that as an element in $R\MMod^{\underline{\bbN}}$, $R[t]/(t-r)$ takes the form $(\mathbbm{1}_{R\MMod},r)$. We make a brief interlude in Section 6 to discussed closed monoidal derivators, the derivator analogue of a closed monoidal category. 

Lastly, we can motivate the \emph{universal property of $\bbA^1$} as follows. The universal property of $R[t]$ over $R$ is easy to describe; let $S$ be another ring and $f_0\colon R\rightarrow S$ be a ring homomorphism. We can ask how to characterize ring homomorphisms $f\colon R[t]\rightarrow S$, and this simply involves a choice of morphism $f:R\rightarrow S$ and element $s\in S$, as we sent $t\in R[t]$ to $s\in S$, and then $f$  extends via linearity to all of $R[t]$, or equivalently all ring homomorphisms, $f\colon R[t]\rightarrow S$ can be characterized by their composition with $R\hookrightarrow R[t]$ and the choice of $f(t)=s\in S$. Here we prove an analoguous characterization for derivators. The two theorems encapsulating the main results are Theorem $\eqref{maintheorem1}$ and Theorem $\eqref{maintheorem2}$. 

Thus, we have defined a derivator $\bbA^1_{\bbD}$ associated to $\bbD$, and we verify that if under reasonable assumptions on $\bbD$, that it has all the good properties of an \emph{affine line} in the sense of $\bbA^1$ in algebraic geometry.

\section{Introduction to derivators}

Here we give only a brief discussion of the necessary definitions and theorems. For a discussion of derivators in generality, Grothendieck defined derivators in his manuscript \cite{Grothendieck91} and Groth's book project \cite{Groth16} gives a very detailed discussion, while \cite{Groth13} gives a compact introduction to the basic theory. 

We first define the notion of prederivator.

\begin{definition}
A prederivator $\bbD$ is a strict 2-functor $\bbD\colon\Cat^{\op} \too \CAT$.
\end{definition}

Here $\Cat$ is the 2-category of small categories, $\CAT$ the 2-category of all categories. The $\op$ is encoding the fact that a prederivator $\bbD$ reverses the direction of the 1-morphisms, that is if we have a prederivator $\bbD$ and a functor $u:I\rightarrow J$, then $$\bbD(u)=u^*\colon\bbD(J)\rightarrow\bbD(I).$$ For the 2-morphisms, given two functors $u,v:I\rightarrow J$ and a natural transformation $\alpha\colon u\rightarrow v$, we have an induced natural transformation in the same direction, $\alpha^*\colon u^*\rightarrow v^*$.

For a functor $u\colon A\rightarrow B$, we call the functor $u^*$ \emph{restriction along u} or \emph{pullback along u}. If in particular if $e$ denotes the terminal category with one object and the identity morphism, and $a\colon e\rightarrow A$ is the functor that sends the single object in $e$ to $a\in A$, then we call $a^*$ the \emph{value at a} functor. For $X\in\bbD(A)$ and $a\in A$, sometimes we may write $X_a$ for $a^*X$. For a prederivator $\bbD$, the category $\bbD(e)$ is called the \emph{underlying category} or \emph{base} of $\bbD$.

\begin{example}
Let $\cat C$ be any (possibly large) category. The \emph{represented prederivator} of $\cat C$ is defined to be the 2-functor $y_{\cat C}\colon\Cat^{\op}\too \CAT$ to take $I\mapsto \cat C^I$, with the usual pullbacks/natural transformations for functors and natural transformations.
\end{example}

We call objects in $\bbD(A)$ \emph{coherent diagrams of shape A}, to distinguish them from \emph{incoherent} diagrams of shape A, which are objects in $\bbD(e)^A$. However, coherent and incoherent diagrams are connected by the so-called ``partial underlying diagram functor defined as follows:

\begin{definition}
Let $\bbD$ be a prederivator, $I,J$ be small categories. The \emph{partial underlying diagram functor} is the functor $$\textrm{dia}_{J,I}\colon \bbD(I\times J)\rightarrow \bbD(I)^J,$$ which sends $X\in\bbD(I\times J)$ to the functor $f_X\colon J\rightarrow\bbD(I)$, where $f_X(j)=j^*X$ for each $j\in J$, and if $\alpha\colon j_1\rightarrow j_2$ is a morphism in $J$, then we get a natural transformation $\alpha^*\colon j_1^*\rightarrow j_2^*$, and so we define $f_X(\alpha):=\alpha^*(X):j_1^*X\rightarrow j_2^*X$.
\end{definition}

In particular, if $I=e$, this means we can think of an object of $\bbD(J)$ as a $J$-shaped diagram in $\bbD(e)$ via the partial underlying diagram functor. This functor is almost never an equivalence, nor does it need to be full/faithful/essentially surjective, but it still assists us in getting some intuition on objects of $\bbD(I)$ as $I$-shaped diagrams in $\bbD(I)$.

Next we give the definition of a derivator.

\begin{definition}
A derivator is a prederivator $\bbD\colon \Cat^{\op}\too\CAT$ satisfying the following conditions.

\begin{enumerate}
\item [Der1:] $\bbD\colon \Cat^{\op}\too\CAT$ takes finite coproducts to products, \ie $$\bbD(J_1\coprod J_2)\cong\bbD(J_1)\times\bbD(J_2).$$ In particular, $\bbD(\varnothing)$ is the terminal category.

\item [Der2:] For any $A\in\Cat$, a morphism $f\colon X\rightarrow Y$ is an isomorphism in $\bbD(A)$ if and only if the morphisms $$a^*f\colon a^*X\rightarrow a^*Y$$ are isomorphisms in $\bbD(e)$ for all $a\in A$.

\item [Der3:] For each functor $u\colon A\rightarrow B$, $u^*\colon \bbD(B)\rightarrow\bbD(A)$ has a left adjoint $u_!$ and a right adjoint $u_*$. $u_!$ and $u_*$ are also referred to as the (left/right) \emph{homotopy Kan extensions} along $u$.

\item [Der4:] For any functor $u\colon A\rightarrow B$ and any object $b\in B$, let us identify $b$ with the functor $b\colon e\rightarrow B$. We have a natural transformation given by 

\smallskip

\centerline{
\xymatrix{
(u/b) \ar[r]^{\textrm{pr}} \ar[d]_{\pi} &
A \ar[d]^{u}\ar@{}[dl]|\swtrans \ar@{}[dl]<-1.0ex>|\alpha\\
e \ar[r]_{b}&
B }}

\noindent Here $(u/b)$ is the slice category with objects, $(a\in A,f\colon u(a)\rightarrow b)$, and morphism $$(a,f)\rightarrow (a',f')$$ given by a morphism $g:a\rightarrow a'$ in $A$ such that $f'\circ u(g)=f$. $pr$ is the forgetful functor sending $(a,f)\mapsto a$, and $\pi$ the projection to $e$. Here the natural transformation $\alpha$ is constructed as follows: for an object $f:u(a)\rightarrow b$, the composition $u\circ\textrm{pr}$ sends it to $u(a)$, while the composition $b\circ\pi$ sends it to $b$. Thus the natural transformation $\alpha$ on the object $(f\colon u(a)\rightarrow b)$ is given by the morphism $f\colon u(a)\rightarrow b$ in $B$. By definition of $(u/b)$ it is easy to see that this patches to a natural transformation $u\circ\textrm{pr}\rightarrow b\circ\pi$. 

After applying $\bbD$, since $u^*$ and $\pi^*$ have left adjoints $u_!$ and $\pi_!$ respectively, we obtain the diagram.

\centerline{
\xymatrix{
\bbD(e)  &  \bbD(u/b) \ar[l]_{\pi_!} \ar@{}[dl]|\swtrans \ar@{}[dl]<-1.0ex>|\epsilon &
\bbD(A)\ar[l]_{\textrm{pr}^*} \ar@{}[dl]|\swtrans \ar@{}[dl]<-1.0ex>|{\alpha^*} & \ar@{}[dl]|\swtrans \ar@{}[dl]<-1.0ex>|\eta \\
 & \bbD(e)\ar[u]^{\pi^*}\ar@/^1.0pc/[ul]^{\textrm{Id}}&
\bbD(B)\ar[u]^{u^*} \ar[l]^{b^*} & \bbD(A) \ar[l]^{u_!} \ar@/_1.0pc/[ul]_{\textrm{Id}} }}

\noindent The morphisms in the two triangles are induced by the unit/counit transformations, respectively. The combined transformation is a morphism $$\textrm{Hocolim}_{(u/b)}\circ \textrm{pr}^*\rightarrow b^*\circ u_!.$$ We require this to be an isomorphism. Similarly, we have a diagram

\centerline{
\xymatrix{
(b/u) \ar[r]^{\textrm{pr}} \ar[d]_{\pi} &
A \ar[d]^{u}\\
e \ar[r]_{b}\ar@{}[ur]|\netrans \ar@{}[ur]<-1.0ex>|\alpha&
B }}

\noindent and a similar natural transformation $$b^*u_*\rightarrow \textrm{Holim}_{(b/u)} pr^*,$$ which we also require to be an isomorphism.
\end{enumerate}
\end{definition}

In particular, from (Der1) and (Der3) we see that $\bbD(A)$ has all finite coproducts and products and therefore also have initial and final objects.

\begin{example}
The following are some examples of derivators.
\begin{enumerate}
\item Let $\cat C$ be a complete and cocomplete category. The represented prederivator $y(\cat C)\colon J\mapsto\cat C^J$ is a derivator.
\item Let $\cat M$ be a model category and $\cat W$ the subcategory of weak equivalences. Then we have the \emph{homotopy derivator} $$\HO(\cat M,\cat W)\colon I\mapsto \cat M^I[(\cat W^I)^{-1}].$$ This is a theorem of Cisinski,\cite[Theorem 1]{Cisinski03}.
\item Let $\cat A$ be a Grothendieck abelian category. We can associate a derivator $\bbD_{\cat A}$ to $\cat A$ by $\bbD_{\cat A}\colon I\mapsto\mathcal{D}(\cat A^I)$, where $\mathcal{D}$ denotes the derived category.
\end{enumerate}
\end{example}

Since for any Grothendieck abelian category $\cat A$ the category of complexes $\textrm{Ch}(\cat A)$ has a natural model structure, the case of Grothendieck abelian categories is subsumed by the general model category case. Nevertheless, for the derivators associated to Grothendieck abelian categories as in case (3), their values are triangulated categories. Such derivators whose values are triangulated are called \emph{strong stable derivators}, and they are an important object of study in their own right. Strong stable derivators have a big advantage over usual triangulated categories, because both the suspension $\Sigma$ and class of distinguished triangles are determined by the natural structure of the derivator, and this makes cones functorial. For more information, see \cite[\textsection 4]{Groth13}, where given a strong stable derivator Groth constructs the natural triangulated structure.

While here we don't need our derivators to be \emph{strong stable}, the notion is undoubtedly important in general. However we will consider the notion of \emph{monoidal derivator}. Putting together the monoidal and triangulated sides of the equations gives us so-called \emph{tensor triangular categories}, which can be studied with geometric methods, see \cite{Balmer10} for more information.

There is one basic operation on derivators that we need to define, namely the ``shift'' of a derivator by any small category.

\begin{proposition}
Let $\bbD$ be a derivator and $L$ be a small category. Then the prederivator $$\bbD^L(I):=\bbD(L\times I)$$ is also a derivator.
\end{proposition}

We direct the reader to \cite[Theorem 1.31]{Groth13} for the proof of this result.

Next we must define the notion of a \emph{homotopy exact square}.  

\begin{definition}
\label{defn:hexact}%
Consider a square with natural transformation: 

\centerline{
\xymatrix{
D \ar[r]^{v_2} \ar[d]_{u_1}&
A \ar[d]^{u_2}\ar@{}[dl]|\swtrans \ar@{}[dl]<-1.0ex>|\alpha\\
B \ar[r]_{v_1}&
C }}

It is said to be homotopy exact if for every derivator $\bbD$, the natural transformation below is an isomorphism

\centerline{
\xymatrix{
\bbD(D) \ar[d]_{(u_1)_!} &
\bbD(A) \ar[l]_{v_2^*}\ar[d]^{(u_2)_!}\\
\bbD(B)&
\bbD(C) \ar[l]^{v_1^*}\ar@{}[ul]|\nwtrans \ar@{}[ul]<-1.0ex>|{\alpha_!} }}

That is to say, we have a natural isomorphism $(v_1)^*(u_2)_!\cong (u_1)_! v_2^*$. 
\end{definition}

\cite[\textsection 1.2]{Groth13} and \cite[\textsection 3]{GrPoSh14a} contain more in-depth discussions of how to check whether squares are homotopy exact. We mention one technical theorem that we will employ repeatedly in this paper.

\begin{definition}
Let $A$ be a small category. Call $A$ \emph{homotopy contractible} if the counit $$(\pi_A)_!(\pi_A)^*\rightarrow\textrm{Id}_{\bbD(e)}$$ is an isomorphism for all derivators $\bbD$ where $\pi_A$ is the projection $A\rightarrow e$.
\end{definition}

Generally one can check whether a category $A$ is homotopy contractible by verifying that its nerve is contractible. More direct methods involve checking whether it has an initial or terminal object, or whether it can be connected via a zigzag of adjunctions to the terminal category $e$.

\begin{definition}
Consider a homotopy exact square as in Definition \ref{defn:hexact}. Let $\gamma$ be a morphism in $C$, $a\in A$ and $b\in B$ be objects. Define the category $(a/D/b)_{\gamma}$ to have objects triples $(d,f\colon a\rightarrow u_1(d),g\colon u_2(d)\rightarrow b)$ such that $v_1(g)\circ\alpha(d)\circ v_2(f)=\gamma$. 

The morphisms between two triples $$(d,f\colon a\rightarrow u_1(d), g\colon u_2(d)\rightarrow b)\rightarrow (d', f'\colon a\rightarrow u_1(d'), g'\colon u_2(d')\rightarrow b)$$ are morphisms $h\colon d\rightarrow d'$ in $D$ such that $u_1(h)\circ f=f'$ and $g'\circ u_2(h)=g$. 
\end{definition}

\begin{theorem}\cite[Theorem 3.8]{GrPoSh14a}
Consider a homotopy exact square 

\centerline{
\xymatrix{
D \ar[r]^{u_1} \ar[d]_{u_2} &
A \ar[d]^{v_2}\ar@{}[dl]|\swtrans \ar@{}[dl]<-1.0ex>|\alpha\\
B \ar[r]_{v_1}&
C }} as in  Definition \ref{defn:hexact}. The square is homotopy exact if and only if for all morphisms $\gamma\in C$ and objects $a\in A$ and $b\in B$, the category $(a/D/b)_{\gamma}$ is homotopy contractible.
\end{theorem}

We direct the reader to \cite{GrPoSh14a} for the proof. Frequently, we will appeal to this theorem to check that squares are homotopy exact.

Lastly, we define sieves and cosieves, which are important for defining the non-monoidal \emph{evaluation at 0} morphism.

\begin{definition}
Let $u\colon I\rightarrow J$ be a fully faithful functor that is injective on objects. 
\begin{enumerate}
\item Call $u$ a \emph{cosieve} if whenever we have a morphism $u(i)\rightarrow j$, then $j$ lies in the image of $u$.
\item Call $u$ a \emph{sieve} if whenever we have a morphism $k\rightarrow u(i)$, then $k$ lies in the image of $u$.
\end{enumerate}
\end{definition}

Homotopy left Kan extensions along cosieves and homotopy right Kan extensions along sieves are simple:

\begin{proposition}\cite[Prop 1.29]{Groth13}
Let $\bbD$ be a derivator.
\begin{enumerate}
\item Let $u\colon I\rightarrow J$ be a cosieve. Then the homotopy left Kan extension $u_!$ is fully faithful, and $X\in\bbD(J)$ lies in the essential image of $u_!$ if and only if $X_j\cong\varnothing$ for all $j\in J-u(I)$.
\item Let $u\colon I\rightarrow J$ be a sieve. Then the homotopy right Kan extension $u_*$ is fully faithful, and $X\in\bbD(J)$ lies in the essential image of $u_*$ if and only if $X_j\cong*$ for all $j\in J-u(I)$.
\end{enumerate}
\end{proposition}

These notions will become relevant for our discussion of the non-monoidal evaluation at 0 morphism. For the latter, it would also make sense to have the notion of a \emph{zero map} and hence also a notion of \emph{pointed} derivator.

\begin{definition}
We say a derivator is $\bbD$ is pointed if $\bbD(e)$ is pointed, \ie if it contains a zero object.
\end{definition}

If $\bbD(e)$ is a pointed category, that is sufficient to make $\bbD(I)$ a pointed category for every $I\in\Cat$ and also to make $u^*,u_!,u_*$ pointed functors for any functor $u\colon I\rightarrow J$ in $\Cat$.

Lastly, we discuss \emph{morphisms of derivators}.

\begin{definition}
A \emph{morphism of prederivators} $F\colon \bbD\rightarrow\bbE$ is a pseudonatural transformation of 2-functors. This means that for each $I\in\Cat$, we have a functor $F_I\colon \bbD(I)\rightarrow\bbE(I)$, and for every $u\colon A\rightarrow B$ in $\Cat$, we have a \emph{chosen} natural isomorphism $$\gamma_u^F\colon u^*F_B\rightarrow F_A u^*,$$ encoded in the following diagram with the usual coherence data.

\centerline{
\xymatrix{
\bbD(B) \ar[r]^{F_B} \ar[d]_{u^*} &
\bbE(B) \ar[d]^{u^*}\ar@{}[dl]|\swtrans \ar@{}[dl]<-1.75ex>|{\gamma_u^F}\\
\bbD(A) \ar[r]_{F_A}&
\bbE(A)  }}

Here, both the functors $F_I$ and the natural transformations $\gamma_u^F$ are part of the data of the morphism of prederivators.
\end{definition}

This is what we would \emph{technically} call a \emph{strong morphism of prederivators}. There are similar notions of \emph{lax morphism} and \emph{strict morphism} of prederivators with $\gamma_u^F$ merely being natural transformations or identities, respectively. Some authors prefer to restrict their attention to strict morphisms, but the inclusion of strong morphisms for this discussion is absolutely essential. 

Here, let PDER denote the 1-category of prederivators and strong morphisms. For us, a \emph{morphism} will mean a strong morphism.

\begin{definition}
A morphism of derivators is simply a morphism of prederivators, except that the source and target are derivators instead of merely prederivators.
\end{definition}

Important classes of morphisms of derivators are given by the following:

\begin{example}
For $\bbD$ a derivator, $A,B\in\Cat$ and a functor $u\colon A\rightarrow B$, there is a \emph{strict} morphism of derivators $$u^*\colon\bbD^B\rightarrow\bbD^A,$$ which is defined at $I\in Cat$ by $$(u\times 1_I)^*\colon\bbD^B(I)=\bbD(B\times I)\rightarrow\bbD(A\times I)=\bbD^A(I).$$

Of course, if $\bbD$ is a derivator, then for a functor $u\colon A\rightarrow B$ in $\Cat$ we also have left and right Kan extensions $u_!,u_*$. We can similarly define morphisms of derivators $u_!,u_*\colon\bbD^A\rightarrow\bbD^B$. These are not strict.
\end{example}

The natural transformation $\gamma_u^F$ also induces natural transformations $$\gamma_{u_!}^F\colon u_!F_A\rightarrow F_Bu_!, \gamma_{u_*}^F\colon F_Bu_*\rightarrow u_*F_A$$ by composing with the unit and counit transformations as needed.

\begin{definition}
Let $F:\bbD\rightarrow\mathbb{E}$ be a morphism of derivators, $u:A\rightarrow B$ is a functor. $F$ \emph{preserves homotopy left Kan extensions along $u$} if the natural transformation $\gamma_{u_!}^F$ is an isomorphism, and $F$ \emph{preserves homotopy right Kan extensions along $u$} if the natural transformation $\gamma_{u_*}^F$ is an isomorphism.

$F$ is \emph{cocontinuous} if it preserves all homotopy left Kan extensions and $F$ is \emph{continuous} if it preserves all homotopy right Kan extensions.
\end{definition}

\begin{example}
For a functor $u\colon A\rightarrow B$, the morphism of derivators $u^*\colon\bbD^B\rightarrow\bbD^A$ is a strict morphism of derivators that is both continuous and cocontinuous.

The morphism of derivators $u_!\colon\bbD^A\rightarrow\bbD^B$ is a cocontinuous strong morphism of derivators, while $u_*\colon\bbD^A\rightarrow\bbD^B$ is a continuous strong morphism of derivators.
\end{example}

We say two derivators $\bbD$, $\mathbb{E}$ are equivalent if there is a morphism $F\colon\bbD\rightarrow\bbE$ such that $F_I$ is an equivalence for all $I\in\Cat$.

\subsection{Notation}

We fix some recurring notation. Let $\underline{\bbN}$ denote the category with one object and endomorphism monoid $\bbN=(\{0,1,2,3,\ldots\},+)$. In picture form,

\begin{equation}
\label{eq:loop}%
\Loop=
\xymatrix{\bullet \ar@(ur,dr)[]^-{\bbN}}
\end{equation}

The functor $i\colon e\rightarrow\underline{\bbN}$ takes the single object in $e$ to the single object in $\underline{\bbN}$, while we denote the projection $\underline{\bbN}\rightarrow e$ as $p$.

Let $[1]$ denote the poset $0<1$, and let $u\colon [1]\rightarrow\underline{\bbN}$ denote the functor that takes $0,1\rightarrow\bullet$ and the morphism $0\rightarrow 1$ to $1\in\textrm{Hom}(\bullet,\bullet)=\bbN$. Let $\ulcorner$ denote the poset category with its coordinates labelled

$$\begin{CD}
(0,0)@>>> (1,0)\\
@VVV\\
(0,1)
\end{CD}$$

We let $i_{[1]}\colon[1]\rightarrow\ulcorner$ be the functor that takes $0\rightarrow 1$ to $(0,0)\rightarrow (1,0)$. Let $\square$ denote the poset with coordinates labelled 

$$\begin{CD}
(0,0)@>>>(1,0)\\
@VVV @VVV\\
(0,1)@>>>(1,1)
\end{CD}$$

Let $i_{\ulcorner}$ denote the obvious inclusion of $\ulcorner$ into $\square$. Here we note that both $i_{[1]}$ and $i_{\ulcorner}$ are both sieves.

Lastly, consider the category $\underline{\bbN}\times\underline{\bbN}$. This is a category with one object, $\bullet$, and endomorphism monoid $\bbN\times\bbN$ with composition being coordinate-wise addition. Define the functor $+\colon \underline{\bbN}\times\underline{\bbN}\rightarrow\underline{\bbN}$ by taking $\bullet\in\underline{\bbN}\times\underline{\bbN}$ to $\bullet\in\underline{\bbN}$, and an element $(a,b)$ of the endomorphism monoid $\bbN\times\bbN$ to $a+b$ in the endomorphism monoid $\bbN$. As mentioned previously, this can also be obtained by thinking of $\underline{\bbN}$ as a symmetric monoidal category; then $+$ is the (symmetric) tensor product.

\section{Canonical morphisms between the base and the affine line}

Let $R$ be a commutative ring with unit. Remember then that $\bbA^1_{\textrm{Spec } R}$ is just $\textrm{Spec } R[t]$. We have three ``natural'' morphisms of rings and/or their spectra given by:

\begin{enumerate}
\item $R\hookrightarrow R[t]$, inducing the structure morphism $$\bbA^1_{\textrm{Spec } R}\rightarrow\textrm{Spec } R$$
\item $R[t]\rightarrow R[t]/(t)\cong R$, inducing the evaluation at 0 morphism $$\textrm{Spec } R\rightarrow\bbA^1_{\textrm{Spec } R}$$
\item $R[t]\rightarrow R[t]/(t-1)\cong R$, inducing the evaluation at 1 morphism $$\textrm{Spec }R\rightarrow\bbA^1_{\textrm{Spec } R}$$
\end{enumerate}

Precisely, we are applying the $\Spec$ functor to those ring homomorphisms above to get the requisite morphisms of affine schemes. We note that for a quasi-compact, quasi-separated scheme $\cat X$ with affine cover $\cat X=\cup_i\textrm{Spec } A_i$, we can take the corresponding cover for $\bbA^1_{\cat X}$ as $\cup_i\textrm{Spec } A_i[t]$. 

\begin{definition}
Let $R$ be a ring. The prederivator $\bbD_R\colon\Cat\too\CAT$ taking $I\mapsto D(R\MMod^I)$ is a derivator.  $R\MMod$ is Grothendieck abelian, and thus also $(R\MMod)^I$; then we take derived categories. Its base is the derived category of $R$, and we call it the \emph{derivator extending the derived category of $R$}.

Let $X$ be a scheme. The prederivator $\bbD_X\colon\Cat\too\CAT$ taking $I\mapsto D(QCoh(X)^I)$ is a derivator. Similarly, we know already that $\textrm{QCoh}(X)$ is Grothendieck abelian, and thus also $(\textrm{QCoh}(X))^I$ is Grothendieck abelian for any small category $I$. Its base is the derived category of quasi-coherent sheaves on $X$, and we call the \emph{derivator extending the derived category of $X$}. 
\end{definition}

Recall that if $X$ is further \emph{separated}, the derived category of quasi-coherent sheaves, $D(\textrm{QCoh}(X))$, is isomorphic to the usual derived category with quasi-coherent cohomology, $D_{\textrm{QCoh}}(X)$, see \cite[Corollary 5.5]{BokstedtNeeman93}. So with the mild addditional condition of separated-ness, we end up working with the ``usual'' derived categories of our scheme $X$.

Here we have equivalences of derivators between $\bbD_R$ and $\bbD_{\textrm{Spec R}}$ for a ring $R$, owing to the isomorphism between $\textrm{QCoh}(\textrm{Spec} R)$ and $R\MMod$.

Moreover, a ring homomorphism $f:R\rightarrow S$ induces morphisms on their corresponding derivators via derived extension of scalars along $f$ and restriction of scalars along $f$. Similarly, given a morphism $g:X\rightarrow Y$ of schemes, there are morphisms of derivators induced by the derived direct and inverse image functors along $g$. With $R$ and $R[t]$, we can describe some of these functors in a diagrammatic manner.

The above ring homomorphisms between $R$ and $R[t]$ generate morphisms between the categories $R\MMod$ and $R[t]\MMod$ via extension of scalars, which extend to morphisms of derivators. These will be our models for the evaluation at 0, evaluation at 1, and structure morphisms.

We first define the structure morphism.

\begin{definition}
The structure morphism of a derivator $\bbD$ and its affine line $\bbA^1_{\bbD}$ is given by the left Kan extension morphism $i_!$ of derivators $$i_!\colon\bbD\rightarrow\bbD^{\underline{\bbN}}.$$
\end{definition}

On affine schemes, the structure morphism is the map $\textrm{Spec} f\colon\bbA^1_R\rightarrow\textrm{Spec} R$ induced by $f\colon R\hookrightarrow R[t]$. Then $i^*$ is also just the direct image $f_*$, while $i_!$ is the inverse image functor $f^*$, so in the case of derivators associated to affine schemes, our definition of structure morphism extends the usual definition of structure morphism.

In particular, we should keep in mind the following diagram:

\centerline{
\xymatrix{
\bbD_R\ar[r]^{i_!} \ar[d]_{1}&\bbA^1_{\bbD_R} \ar[d]^{\cong}\\
\bbD_R \ar[r]_{\textrm{Spec} f^*}& \bbD_{R[t]} }} which explains that the choice of $i_!$ is indeed appropriate. 


Next, we have the evaluation at 1 morphism. It is also a homotopy left Kan extension.

\begin{definition}
The ``evaluation at 1'' morphism relating a derivator $\bbD$ and its affine line $\bbA^1_{\bbD}$ is given by the homotopy left Kan extension morphism of derivators $$p_!\colon\bbD^{\underline{\bbN}}\rightarrow\bbD.$$ 
\end{definition}

We can see this in case of an affine scheme $\textrm{Spec} R$, when we have the evaluation at 1 map $\textrm{Spec} f^*\colon\textrm{Spec } R\rightarrow\bbA^1_R$ induced by the ring homomorphism $$f\colon R[t]\rightarrow R[t]/(t-1)\cong R$$ The map $p^*\colon R\MMod\rightarrow R\MMod^{\underline{\bbN}}=R[t]\MMod$ is the direct image functor $f_*$ and $p_!$ is the inverse image functor $f^*$. So the evaluation at 1 map, in the case of an affine scheme, can be thought of as induced by the usual scheme-theoretic ``evaluation at 1'' map.

Again, this  should be envisioned in the below diagram: 

\centerline{
\xymatrix{
\bbA^1_{\bbD_R}\ar[r]^{i_!} \ar[d]_{1}&\bbD_R \ar[d]^{\cong}\\
\bbD_{R[t]} \ar[r]_{\textrm{Spec} f^*}& \bbD_R}}

Later when we have the machinery of monoidal derivators, we will be able to give a unified definition of evaluation at 0 and 1 along with other ``coherent endomorphisms.'' For now, we will have to stick with a somewhat unwieldy definition for evaluation at 0.
Let us assume that the derivator $\bbD$ is now pointed, so that the right Kan extension $i_{[1] *}$ is just extension by zero. 

Given a derivator $\bbD$ and $X\in\bbD(\underline{\bbN})$ we can apply pullback via the functor $u\colon [1]\rightarrow\underline{\bbN}$ to obtain $u^*(X)\in\bbD([1])$. 

From there, we include $i_{[1]}\colon[1]\rightarrow\ulcorner$ and $i_{\ulcorner}\colon\ulcorner\rightarrow\square$. In terms of underlying diagrams, the composition $i_{\ulcorner !}i_{[1] *}u^*$ takes an element $(M,f)$ in $\bbD^{\underline{\bbN}}$ to the (coherent) homotopy pushout square (i.e. element of $\bbD(\square)$)  below. As $i_{[1]}$ is a sieve, the bottom left corner of the coherent diagram is the 0 object. Left Kan extension along $i_{\ulcorner}$ is just making a homotopy pushout square. Identifying the bottom right corner gives us the proposed evaluation at zero map$\text{ev}_0=(1,1)^*i_{\ulcorner !}i_{[1] *} u^*$. Here the underlying diagram looks like

\centerline{
\xymatrix{
M\ar[r]^{f}\ar[d]& M\ar[d]\\
0\ar[r] & M/f(M)
}}

Each operation in the composition is a morphism of derivators, hence also $\textrm{ev}_0$. We may simplify this one step further. The composition $(1,1)^*i_{\ulcorner !}$ is actually just the homotopy colimit of the $\ulcorner$-shaped diagram, and we can write it as $\pi_{\ulcorner !}$. 

\begin{proposition}
The evaluation at 0 morphism is $\text{ev}_0=(\pi_{\ulcorner})_!i_{[1] *} u^*$.
\end{proposition}

We see also that this fits the ``evaluation at 0'' morphism for affine schemes, in that case being the map $\textrm{Spec} f\colon \textrm{Spec} R\rightarrow\bbA^1_R$, induced by the homomorphism $f\colon R[t]\rightarrow R[t]/(t)$. We see directly that the construction emulates $$M\mapsto M\otimes_{R[t]} R[t]/(t)$$ for an $R[t]$-module $M$.

Specifically we are referring to the following diagram that commutes up to natural isomorphism: here $\bbD_R$ is the derivator associated to the ring $R$ as usual; 

\centerline{
\xymatrix{
\bbA^1_{\bbD_R}\ar[r]^{i_!} \ar[d]_{1}&\bbD_R \ar[d]^{\cong}\\
\bbD_{R[t]} \ar[r]_{\textrm{Spec} f^*}& \bbD_R}}


\bigskip

\section{Monoidal structures}

\subsection{Definitions on monoidal derivators}

As our intuition comes from algebraic geometry, in many situations our derivators $\bbD$ will have a monoidal structure. An important consideration would be to construct a monoidal structure on $\bbA^1_{\bbD}$ out of the monoidal structure on $\bbD$ and information internal to $\bbD$. If $\bbD$ were the derivator associated to a ring $R$, the monoidal structure on $\bbD$ would be induced by $-\otimes_R -$, while on $\bbA^1_{\bbD}$ the ``correct'' monoidal structure should be induced by $-\otimes_{R[t]} -$. In the introduction, we briefly described how we can relate the two monoidal structures by adding in what was essentially a coequalizer. Here we will construct the monoidal structure in a more formal manner, using pullbacks and left Kan extensions.

First, we give the basic definitions around monoidal derivators. We lift much of the exposition below from \cite{GrPoSh14b}.

\begin{definition}\cite[Definition 3.1]{GrPoSh14b}
A monoidal prederivator is a pseudomonoid object in $\textrm{PDER}$.

This means that $\bbD$ is a monoidal prederivator if there is a product $$\otimes\colon\bbD\times\bbD\rightarrow\bbD$$ that gives compatible monoidal structures on each $\bbD(A)$, i.e. for any functor $u\colon A\rightarrow B$, the below diagram commutes-

\centerline{
\xymatrix{
\bbD(B)\times\bbD(B) \ar[r]^(0.6){\otimes_B} \ar[d]_{u^*\times u^*}&
\bbD(B) \ar[d]^{u^*}\ar@{}[dl]|\swtrans \ar@{}[dl]<-1.5ex>|{Id}\\
\bbD(A)\times\bbD(A)\ar[r]_(0.6){\otimes_A}&
\bbD(A) }}
 
Equivalently, we can say that there is a lift of $\bbD$ to the 2-category of monoidal categories and strong monoidal functors. Here we denote the monoidal structure on $\bbD(A)$ as $\otimes_A$. 
\end{definition}

Let $\bbD$ be a monoidal prederivator, $A\in\Cat$ and $a\in A$, then for $X,Y\in\bbD(A)$, $(X\otimes_A Y)_a=X_a\otimes Y_a$ by the commutativity of the above diagram.

\begin{definition}
The external product on a monoidal prederivator $(\bbD,\otimes)$ is defined as follows. Let $A,B$ be two small categories, $\pi_B\colon A\times B\rightarrow A$, $\pi_A\colon A\times B\rightarrow B$ be the projection functors onto each component. The external product $\boxtimes_{\bbD}$ is the following composition:

$$\begin{CD}
\bbD(A)\times\bbD(B)@>\pi_B^*\times\pi_A^*>>\bbD(A\times B)\times\bbD(A\times B)@>\otimes_{A\times B}>>\bbD(A\times B)
\end{CD}$$

\noindent This allows us to take two objects $X\in\bbD(A)$ and $Y\in\bbD(B)$, and define a product of them in $\bbD(A\times B)$. The monoidal structure on $\bbD(A)$ can be recovered from the external product via the formula

$$\begin{CD}
\bbD(A)\times\bbD(A)@>\boxtimes>>\bbD(A\times A)@>\Delta^*_A>>\bbD(A)
\end{CD}$$

where $\Delta_A\colon A\rightarrow A\times A$ is the diagonal functor on $A$, and is sometimes referred to as the \emph{internal product}.
\end{definition}

There are also mixed internal/external products, of the form 

$$\begin{CD}
\bbD(A\times B)\times \bbD(B\times C)@>>>\bbD(A\times B\times C)
\end{CD}$$

taking $$(X,Y)\mapsto (1_A\times\Delta_B\times 1_C)^*(X\boxtimes Y).$$ We should just think of this defining an external product in $\bbD^B$. The crux of the monoidal structure of the monoidal (pre)derivator is the external product. 

Next we define a monoidal derivator. 

\begin{definition}\cite[Definition 3.21]{GrPoSh14b}
A monoidal derivator is a derivator with product making it a monoidal prederivator, and whose product $\otimes\colon\bbD\times\bbD\rightarrow\bbD$ is co-continuous in each variable. More specifically, if $X\in\bbD(A)$, $Y\in\bbD(C)$, where $A,B,C,D$ are categories and $u\colon A\rightarrow B$ is a functor, then we have natural isomorphisms $(u\times 1_C)_!(X\boxtimes Y)\rightarrow (u_!X\boxtimes Y)$ and similarly for a functor $v\colon C\rightarrow D$ we have a natural isomorphism $(1_A\times v)_!(X\boxtimes Y)\rightarrow (X\boxtimes v_!Y)$.
\end{definition}

In terms of diagrams, we have that the following square commutes up to isomorphism:

\centerline{
\xymatrix{
\bbD(A)\times\bbD(C) \ar[r]^(0.5){u_!\times 1} \ar[d]_{\boxtimes} &
\bbD(B) \times\bbD(C) \ar[d]^{\boxtimes}\ar@{}[dl]|\swtrans \ar@{}[dl]<-1.0ex>|\cong\\
\bbD(B)\times\bbD(C)\ar[r]_(0.5){(u\times 1)_!}&
\bbD(B\times C) }}and similarly for the left Kan extension along $v\colon C\rightarrow D$.

Using (Der2) and (Der4), we note that the explicit verification of this cocontinuity condition need only be done for projections $\pi_A\colon A\rightarrow e$.


Now we construct the $\bbA^1_{\bbD}$-monoidal structure on $\bbD^{\underline{\bbN}}$ for a monoidal derivator $\bbD$. The $\bbA^1_{\bbD}$-monoidal structure should specialize to the usual monoidal structures for $\bbD=\bbD_R$ and $\bbA^1_{\bbD}=\bbD_{R[t]}$, extending the derived categories of $R\MMod$ and $R[t]\MMod$, that is to say, derived versions of $-\otimes_R -$ and $-\otimes_{R[t]}-$.

Recall from the introduction that the tensor $-\otimes_{R[t]} -$ is obtained via a sort of coequalizer. We can make the construction more amenable in the derivator context. If we were to take a hypothetical ``external product'' of $(M,t_M)$ and $(N,t_N)$ (thought of as objects in $(R\MMod^{\underline{\bbN}}$ with their external product landing in $R\MMod^{\underline{\bbN}\times\underline{\bbN}}\cong R[t_1,t_2]\MMod$), we would obtain an $R[t_1,t_2]$-module $$(M\otimes_R N,t_1=t_M\otimes 1:M\otimes_R N\rightarrow M\otimes_R N,t_2=1\otimes t_N\colon M\otimes_R N\rightarrow M\otimes_R N)$$ To obtain the result we want of $M\otimes_{R[t]} N$ we would take the extension of scalars along the ring homomorphism $R[t_1,t_2]\rightarrow R[t]/(t_1-t_2)\cong R[t]$.

Note that $+^*\colon R\MMod^{\underline{\bbN}}\rightarrow R\MMod^{\underline{\bbN}\times\underline{\bbN}}$ is the functor $$R[t]\MMod\rightarrow R[t_1,t_2]\MMod$$ where the action of $t$ on the underlying $R$-module is is used for both the actions of $t_1,t_2$ in the $R[t_1,t_2]$-module. But this is precisely restriction of scalars along the ring homomorphism $$R[t_1,t_2]\rightarrow R[t_1,t_2]/(t_1-t_2)\cong R[t]$$ Therefore, the extension of scalars along $R[t_1,t_2]\rightarrow R[t_1,t_2]/(t_1-t_2)$ is the left Kan extension, $+_!$. 

The generalization to derivators is now clear. First we write down the external product formula for $\bbD^{\underline{\bbN}}$ on just $\bbD^{\underline{\bbN}}(e)$. Let $X,Y$ be elements of $\bbD^{\underline{\bbN}}(e)=\bbD(\underline{\bbN})$ and let $(M,f)$ and $(N,g)$ denote their underlying diagrams. Then their external product $(M,f)\boxtimes (N,g)$ in $\bbD(\underline{\bbN}\times\underline{\bbN})$ looks like $(M\otimes N,f\otimes 1,1\otimes g)$. An application of $+_!$ will give precisely what we described above, and coincides with our description if $\bbD=\bbD_R$ for a ring $R$ and $\bbA^1_{\bbD}=\bbD_{R[t]}$. 

Since there are two distinct external products that we will often be using in relation to $\bbA^1_{\bbD}$, we let $\boxtimes_{\bbD}$ denote $\bbD$-external product and $\boxtimes_{\bbA^1_{\bbD}}$ denote $\bbA^1_{\bbD}$-external product. Similarly we let $\otimes_{\bbD}$ denote $\bbD$-tensor products, $\otimes_{\bbA^1_{\bbD}}$ denote $\bbA^1_{\bbD}$-tensor products.

\begin{proposition}
\label{A1monoid}
Let $(\bbD,\boxtimes_{\bbD})$ be a monoidal derivator. Define the external product on $\bbD^{\underline{\bbN}}$ as $X\boxtimes_{\bbA^1_{\bbD}} Y:=+_!(X\boxtimes_{\bbD} Y)$. This makes $\bbA^1_{\bbD}$ a monoidal derivator. Its monoidal unit is $i_!\mathbbm{1}_{\bbD}$. 
\end{proposition}

\begin{proof}
$\bbD^{\underline{\bbN}}$ is certainly a derivator and the external product certainly defines a product morphism $$\bbD^{\underline{\bbN}}\times\bbD^{\underline{\bbN}}\rightarrow\bbD^{\underline{\bbN}},$$ so we merely need to prove that the external product is cocontinuous. 

To see that the external product on $\bbA^1_{\bbD}$ is cocontinuous, let $X\in\bbA^1_{\bbD}(J)$, $\pi_J:J\rightarrow e$ be the projection functor, and $y\in\bbA^1_{\bbD}(e)$. Consider the external products $(\pi_J)_!X\boxtimes_{\bbA^1_{\bbD}} Y$ and $(\pi_J\times 1)_!(X\boxtimes_{\bbA^1_{\bbD}} Y)$. 

Then in terms of $\boxtimes$, we have

\begin{eqnarray*}
((\pi_J)_!X\boxtimes_{\bbA^1_{\bbD}} Y)&=&+_!((\pi_J)_!X\boxtimes_{\bbD} Y)\\
&\cong&+_!((\pi_J\times1_{\underline{\bbN}})_! X\boxtimes_{\bbD} Y)\\
&\cong&+_!(\pi_J\times1_{\underline{\bbN}}\times 1_{\underline{\bbN}})_! (X\boxtimes_{\bbD} Y)\\
&\cong&(\pi_J\times 1_{\underline{\bbN}})_!+_!(X\boxtimes_{\bbD} Y)\\
&\cong& (\pi_J\times 1)_!(X\boxtimes_{\bbA^1_{\bbD}} Y)
\end{eqnarray*}

In particular, the middle equality comes from the commutativity of the diagram

$$\begin{CD}
J\times\underline{\bbN}\times\underline{\bbN}@>1_J\times +>>J\times\underline{\bbN}\\
@V\pi_J\times1_{\underline{\bbN}\times\underline{\bbN}}VV @V\pi_J\times1_{\underline{\bbN}}VV\\
\underline{\bbN}\times\underline{\bbN}@>+>>\underline{\bbN}
\end{CD}$$

\noindent Moreover, the monoidal unit for $\bbA^1_{\bbD}$ with $\boxtimes_{\bbA^1{\bbD}}$  is $i_!\mathbbm{1}_{\bbD}$. First note that the composition 

$$\begin{CD}
\underline{\bbN}@>(i\times 1)\textrm{ or } (1\times i)>>\underline{\bbN}\times\underline{\bbN}@>+>> \underline{\bbN}
\end{CD}$$ is the identity. We have that $+_!(i\times 1)_!\cong \textrm{id}_!\cong +_!(1\times i)_!$, and $\textrm{id}_!$ is an isomorphism (being an adjoint to the identity). Then we have for $Y\in\bbA^1_{\bbD}(I)$,

\begin{eqnarray*}
(i_!\mathbbm{1}\boxtimes_{\bbA^1_{\bbD}} Y)&=&+_!(i_!\mathbbm{1}\boxtimes_{\bbD} Y)\\
&=&+_!(i\times 1)_!(\mathbbm{1}\boxtimes_{\bbD} Y)\\
&\cong &(\mathbbm{1}\boxtimes_{\bbD} Y)\\
&\cong & Y
\end{eqnarray*}
\end{proof}

Note that this monoidal structure is an analogue of Day convolution. We may view $\underline{\bbN}$ as a natural symmetric monoidal category with only one object, where on maps $n\otimes m=n+m\in\bbN$. Then $+\colon \underline{\bbN}\times\underline{\bbN}\rightarrow\underline{\bbN}$ is just expressing the monoidal product and the usage of $+_!$ in the formula for $-\boxtimes_{\bbA^1_{\bbD}}-$ is the usual formula.

If $M$ is a commutative monoid, we let $\underline{M}$ denote its categorical analogue, i.e. the category with one object $\bullet$ and $\textrm{Hom}(\bullet,\bullet)=M$ with composition as addition in $M$. Then again we can view $\underline{M}$ as a symmetric monoidal category where the tensor product of maps is given by addition in $M$, and if $\bbD$ is a symmetric monoidal derivator, then $\bbD^{\underline{M}}$ has the natural structure of a symmetric monoidal derivator by defining the external product as $-\boxtimes_{\bbD^{\underline{M}}}-=(+_M)_!(-\boxtimes -)$. 

We make a brief observation about the monoidal structure on $\bbA^n_{\bbD}$. Recall from \cite{BalmerZhang16} that similar to $\bbA^1$ we may define $\bbA^n_{\bbD}=\bbD^{\underline{\bbN}^n}$. Alternatively, we can recursively define $$\bbA^{n+1}_{\bbD}=\bbA^1_{\bbA^n_{\bbD}},$$ which allows us to define the monoidal structure on $\bbA^n_{\bbD}$ by recurring the $\bbA^1_{\bbD}$-monoidal structure. 

The monoidal structure on $\bbA^n_{\bbD}$ can be expressed with the following external product:

\begin{eqnarray*}
X\boxtimes_{\bbA^n_{\bbD}} Y&:=& (+_{\underline{\bbN}^n})_!(X\boxtimes_{\bbD} Y)
\end{eqnarray*} Here it is not difficult to see that this is the same external product that we would obtain with an iterated $\bbA^1$-structure.


\section{Compatibility of canonical morphisms and monoidal structure}

In this section, we want to define a general \emph{evaluation at $\alpha$} morphism. As mentioned in the introduction, our model for this is the extension by scalars via the homomorphism $R[t]\rightarrow R[t]/(t-r)\cong R$, \ie $$-\otimes_{R[t]} R[t]/(t-r)\colon R[t]\MMod\rightarrow R\MMod$$ 

In the derivator case, these emph{evaluation at $alpha$} morphisms ought to be strong monoidal. Furthermore, the composition $$R\hookrightarrow R[t]\rightarrow R[t]/(t-r)\cong R$$ is the identity. The structure morphism is a model for extension of scalars along $R\hookrightarrow R[t]$, so the evaluation at $\alpha$ should form a section of the structure morphism.

First, we show that the structure morphism $i_!\colon\bbD\rightarrow\bbA^1_{\bbD}$ is a strong monoidal morphism under the $\bbA^1_{\bbD}$-monoidal structure as defined in the previous section.

\begin{proposition}
The structure morphism is strong monoidal, i.e. $i_!(X\boxtimes_{\bbD} Y)\cong (i_!X\boxtimes_{\bbA^1_{\bbD}} i_! Y)$.
\end{proposition}

\begin{proof}
This is a straightforward computation, we have

\begin{eqnarray*}
(i_!X\boxtimes_{\bbA^1_{\bbD}} i_!Y)&=& +_!(i_! X\boxtimes_{\bbD} i_! Y)\\
&\cong& +_!(i\times 1)_! (X\boxtimes_{\bbD} i_! Y)\\
&\cong& +_!(i\times 1)_! (1\times i)_! (X\boxtimes_{\bbD} Y)\\
&\cong& (+\circ (i\times 1)\circ (1\times i))_! (X\boxtimes_{\bbD} Y)\\
&\cong& i_!(X\boxtimes_{\bbD} Y)
\end{eqnarray*}
\end{proof}

The construction of the \emph{evaluation at $\alpha$} morphism strongly mirrors the ring case.

\begin{definition}
Let $\bbD$ be a derivator, and $X$ be an object of $\bbD(\underline{\bbN})=\bbA^1_{\bbD}(e)$, such that $i^*X=\mathbbm{1}_{\bbD}$. Then call $X$ an \emph{coherent endomorphism of the unit}. We can write $\textrm{dia}_{\underline{\bbN},e}(X)=(\mathbbm{1},\alpha)$. In this case we call $X$ the \emph{coherent $\alpha$ endomorphism}.
\end{definition}

The term \emph{coherent $\alpha$ endomorphism} can be a bit deceptive, as there may be more than one object with the same underlying diagram. Now we can define what the \emph{evaluation at $\alpha$} morphism means.
 
Let $\bbD$ be a symmetric monoidal derivator and equip $\bbD^{\underline{\bbN}}$ with the $\bbA^1_{\bbD}$-monoidal structure. We would take a coherent endomorphism of the unit, i.e. an object in $\bbD(\underline{\bbN})$ with diagram $(\mathbbm{1},\alpha)$, take its $\bbA^1_{\bbD}$-external product any object in some $\bbA^1_{\bbD}(I)$, and then forget the endomorphism part originating in the $(\mathbbm{1},\alpha)$-direction. That is to say, we take $X\in\bbD^{\underline{\bbN}}(I)$ to $i^*+_!(X\boxtimes_{\bbD} (\mathbbm{1},\alpha))$.

We call $p^*\mathbbm{1}$ the \emph{coherent identity morphism}, and evaluation at 1 means evaluating at this particular element of $\bbD(\underline{\bbN})$. 

\begin{definition}
Let $\bbD$ be a symmetric monoidal derivator, and consider $\bbA^1_{\bbD}$ with the $\bbA^1_{\bbD}$-monoidal structure. The \emph{evaluation at (coherent) $\alpha$} morphism takes $Y\in\bbA^1_{\bbD}(I)$ to $i^*(Y\boxtimes_{\bbD} (\mathbbm{1}_{\bbD},\alpha))$. 
\label{ev alpha}
\end{definition}

We will show in due course that this is a strong monoidal morphism. First we show that it coincides with our previous notion of evaluation at 1.

\begin{lemma}
The definition of ``evaluation at 1'' given by $p_!$ and the new definition of evaluation at 1 coincide.
\end{lemma}

\begin{proof}
Using general definition of evaluation at 1, our coherent endomorphism is $p^*\mathbbm{1}_{\bbD}$. Thus we have

\begin{eqnarray*}
\textrm{ev}_1(X)&=&i^*+_!(X\boxtimes_{\bbD} p^*\mathbbm{1})\\
&=&i^*+_!(1\times p)^*(X)
\end{eqnarray*}

So the task at hand is now simply to prove the isomorphism $i^*+_!(1\times p)^*\cong p_!$. We will prove a related isomorphism, namely $+_!(1\times p)^*\cong p^*p_!$. The required isomorphism now follows from this one since $i^*p^*\cong\textrm{Id}$, and so post-composing both sides of $+_!(1\times p)^*\cong p^*p_!$ gives precisely $i^*+_!(1\times p)^*\cong p_!$. 
Thus, we would like to show the following (commutative) square is homotopy exact.

\centerline{
\xymatrix{
\underline{\bbN}\times\underline{\bbN} \ar[r]^{1\times p} \ar[d]_{+} &
\underline{\bbN} \ar[d]^{p}\ar@{}[dl]|\swtrans \ar@{}[dl]<-1.5ex>|{\textrm{Id}}\\
\underline{\bbN} \ar[r]_{p}&
e }}

Here we check this directly via \cite[Theorem 3.8]{GrPoSh14a}. In our good fortune, because we have the terminal category in the lower right corner, we need only check that a single category is homotopy contractible. The category $(\bullet/\underline{\bbN}\times\underline{\bbN}/\bullet)_{\textrm{id}}$ as stated in the theorem, where both objects $\bullet$ are the sole objects in the two copies of $\underline{\bbN}$, has objects triples $$(m\in\bbN=\textrm{Hom}_{\underline{\bbN}}(\bullet,\bullet),\bullet\in\underline{\bbN}\times\underline{\bbN}, n\in\bbN=\textrm{Hom}_{\underline{\bbN}}(\bullet,\bullet)),$$ which we view as merely a pair of natural numbers $(m,n)$. 

The morphisms in this category are morphisms $(a_1,a_2)\in\underline{\bbN}\times\underline{\bbN}$ such that $(1\times p)(a_1,a_2)+m=m'$, and $+(a_1,a_2)+n'=n$, \ie we have a morphism $(a_1,a_2):(m,n)\rightarrow (m',n')$ if $a_1+m=m'$ and $a_1+a_2+n'=n$. Between any two objects of this category $(\bullet/\underline{\bbN}\times\underline{\bbN}/\bullet)_{\textrm{id}}$, there is at most only a single morphism. In particular, there is a morphism $(m,n)\rightarrow (m',n')$, if $m\leq m'$ and $n-n'\geq m'-m$, and in particular we must have $m\leq m'$ and $n\geq n'$. 

Thus, we may view $(\bullet/\underline{\bbN}\times\underline{\bbN}/\bullet)_{\textrm{id}}$ as a subcategory of $(\bbN,\leq)\times (\bbN,\geq)$ containing all objects but not all morphisms, where we have a morphism $(m,n)\rightarrow (m',n')$ if and only if $m\leq m'$ and $n\geq n'$, and $m+n\leq m'+n'$. However, we may view this as a subcategory of $(\bbN,\leq)\times (\bbN,\geq)$ by taking $(m,n)\mapsto (m+n,n)$. Being a fully faithful functor, it is an equivalence onto its image, and this second category is a full subcategory of $(\bbN,\leq)\times (\bbN,\geq)$ consisting of objects $(k,l)$ with $k\geq l$. Call this subcategory $L\subset (\bbN,\leq)\times (\bbN,\geq)$. 

There is an adjunction connecting $L$ and $(\bbN,\leq)\times (\bbN,\geq)$. The right adjoint is the inclusion, and the left adjoint takes $(m,n)\in (\bbN,\leq)\times (\bbN,\geq)$ to $(m,n)\in L$ if $m\geq n$ and $(n,n)\in L$ if $m<n$. Then $(\bbN,\leq)\times (\bbN,\geq)$ is a product of two homotopy contractible categories, since $(\bbN,\leq)$ has an initial object while $(\bbN,\geq)$ has a final object, and hence homotopy contractible. 

$L$ is then connected via adjunction to the terminal category, hence it is homotopy contractible. Therefore, our commutative square is homotopy exact and the two definitions of evaluation at 1 coincide.
\end{proof}

\begin{lemma}
Let $Y=(\mathbbm{1},\alpha)$ be a coherent endomorphism in $\bbD(\underline{\bbN})$ and let $M,N\in\bbD(e)$. Then $(M\boxtimes_{\bbD} Y)\boxtimes_{\bbA^1_{\bbD}} (N\boxtimes_{\bbD} Y)\cong (M\boxtimes_{\bbD} N)\boxtimes_{\bbD} Y$. 
\end{lemma}

\begin{proof}
This can be proven by showing that $+_! +^*\cong \textrm{Id}$, or equivalently that the square

\centerline{
\xymatrix{
\underline{\bbN}\times\underline{\bbN}\ar[d]^{1} \ar[r]^{1}& \underline{\bbN}\times\underline{\bbN}\ar[d]^{+}\ar@{}[dl]|\swtrans \ar@{}[dl]<-1.5ex>|{\textrm{Id}}\\
\underline{\bbN}\times\underline{\bbN}\ar[r]^{+}&\underline{\bbN}}} is homotopy exact. However, this statement is certainly true if $\bbD$ is a model category, so by \cite[Theorem 3.16]{GrPoSh14a}, it holds for all derivators $\bbD$.
\end{proof}

\begin{proposition}
The general ``evaluation at $\alpha$'' morphism is a strong monoidal morphism of derivators for any coherent endomorphism $(\mathbbm{1},\alpha)\in\bbD^{\underline{\bbN}}(e)$. 
\end{proposition}

\begin{proof}
Again we simply restrict to examining the evaluation at $\alpha$ morphism on $\bbA^1_{\bbD}(e)$. Consider objects $X$ and $Y$ in $\bbD(\underline{\bbN})$-we wish to show $$\textrm{ev}_{\alpha}(X)\boxtimes_{\bbD}\textrm{ev}_{\alpha}(Y)\cong\textrm{ev}_{\alpha}(X\boxtimes_{\bbA^1_{\bbD}} Y).$$ 

Let us use the formulation of $\textrm{ev}_{\alpha}(-)=i^*+_!(-\boxtimes_{\bbD} (1,\alpha))$. Then we have

\begin{eqnarray*}
\textrm{ev}_{\alpha}(X\boxtimes_{\bbA^1_{\bbD}}Y)&=&i^*+_!((X\boxtimes_{\bbA^1_{\bbD}}Y)\boxtimes_{\bbD} (\mathbbm{1},\alpha))\\
&=&i^*((X\boxtimes_{\bbA^1_{\bbD}}Y)\boxtimes_{\bbA^1_{\bbD}}(\mathbbm{1},\alpha))\\
&\cong&i^*((X\boxtimes_{\bbA^1_{\bbD}}Y)\boxtimes_{\bbA^1_{\bbD}}((\mathbbm{1},\alpha)\boxtimes_{\bbA^1_{\bbD}}(\mathbbm{1},\alpha)))\\
&\cong&i^*(X\boxtimes_{\bbA^1_{\bbD}}(\mathbbm{1},\alpha))\boxtimes_{\bbA^1_{\bbD}}(Y\boxtimes_{\bbA^1_{\bbD}}(\mathbbm{1},\alpha))\\
&\cong&i^*(X\boxtimes_{\bbA^1_{\bbD}}(\mathbbm{1},\alpha))\boxtimes_{\bbD} i^*(Y\boxtimes_{\bbA^1_{\bbD}}(\mathbbm{1},\alpha))\\
&=&(i^*+_!)(X\boxtimes_{\bbD}(\mathbbm{1},\alpha))\boxtimes_{\bbD} (i^*+_!)(Y\boxtimes_{\bbD} (\mathbbm{1},\alpha))\\
&=&\textrm{ev}_{\alpha}(X)\boxtimes_{\bbD}\textrm{ev}_{\alpha}(Y)
\end{eqnarray*}

Here the last isomorphism comes from the preceding lemma.
\end{proof}

Therefore, the monoidal structures on $\bbA^1_{\bbD}$ and $\bbD$ are compatible, in that the structure and evaluation morphisms are all strong monoidal.

\begin{proposition}
The evaluation at $\alpha$ morphisms are cocontinuous.
\end{proposition}

\begin{proof}
To be precise, we would like to show that for any functor $f\colon A\rightarrow B$, the diagram below commutes.

\centerline{
\xymatrix{
\bbA^1_{\bbD}(A) \ar[r]^{f_!} \ar[d]_{\textrm{ev}_{\alpha}} &
\bbA^1_{\bbD}(B) \ar[d]^{\textrm{ev}_{\alpha}}\ar@{}[dl]|\swtrans \ar@{}[dl]<-1.5ex>|{\cong}\\
\bbD(A) \ar[r]_{f_!}&
\bbD(B) }}

Remember that $\textrm{ev}_{\alpha}(X)$ for any $X\in\bbA^1_{\bbD}(I)$ is the composition $i^*+_!(X\boxtimes_{\bbD} (\mathbbm{1},\alpha))$. Taking an external product with $(\mathbbm{1},\alpha)$ commutes with $f_!$ since external products are cocontinuous. The homotopy left Kan extensions $+_!$ and $f_!$ commute as they occur in different variables, and similarly $i^*$ and $f_!$. 
\end{proof}

This last result will be of importance when we discuss the ``universal property'' of the affine line.


\section{Closed monoidal derivators}

We call a monoidal derivator $\bbD$ closed if the tensor product $$\otimes\colon\bbD\times\bbD\rightarrow\bbD$$ is a two-variable left adjoint. For information on general two-variable left adjoints in this context, we defer to \cite[Section 8]{GrPoSh14b}. Our main goal in this section is to prove that if $\bbD$ is a closed symmetric monoidal derivator, then so is $\bbA^1_{\bbD}$ with the $\bbA^1$-monoidal structure as defined in Section 4. 

\begin{theorem}
If $(\bbD,\boxtimes)$ is a closed monoidal derivator, then so is $(\bbA^1_{\bbD},\boxtimes_{\bbA^1_{\bbD}})$. 
\end{theorem}

\begin{proof}
Recall that the external product associated to $\otimes_{\bbA^1}$ is $\boxtimes_{\bbA^1_{\bbD}}:\bbA^1_{\bbD}(I)\times\bbA^1_{\bbD}(J)\rightarrow\bbA^1_{\bbD}(I\times J)$ defined as

$$\begin{CD}
\bbD(\underline{\bbN}\times I)\times\bbD(\underline{\bbN}\times J)@>\boxtimes_{\bbD}>>\bbD(\underline{\bbN}\times I\times\underline{\bbN}\times J)@>(+\times 1)_!>>\bbD(\underline{\bbN}\times I\times J)
\end{CD}$$

To show this is a two-variable left adjoint, we see that

\begin{enumerate}
\item $\rhd_{[\bbA^1,J]}\colon\bbA^1_{\bbD}(J)^{op}\times\bbA^1_{\bbD}(I\times J)\rightarrow\bbA^1_{\bbD}(I)$ is given by the composition $$\rhd_{[\underline{\bbN},J]}\circ (\textrm{Id}\times (+\times 1_{I\times J})^*),$$ which we obtain by taking the right adjoints of $\boxtimes_{\bbD}$ and $(+\times 1)_!$.  
\item $\lhd_{[\bbA^1,I]}\colon\bbA^1_{\bbD}(I\times J)\times\bbA^1_{\bbD}(I)^{op}\rightarrow\bbA^1_{\bbD}(I)$ given by the composition $$\lhd_{[\underline{\bbN}\times I]}\circ ((+\times 1_{I\times J})^*\times\textrm{Id}),$$ again obtained by taking right adjoints of $\boxtimes_{\bbD}$ and $(+\times 1)_!$. 
\end{enumerate}

We should then have, for any $X\in\bbA^1_{\bbD}(I)$, $Y\in\bbA^1_{\bbD}(J)$, and $Z\in\bbA^1_{\bbD}(I\times J)$, natural isomorphisms

$$\bbA^1_{\bbD}(I\times J)(X\boxtimes_{\bbA^1_{\bbD}} Y,Z)\cong\bbA^1_{\bbD}(I)(X,Y\rhd_{[\bbA^1,J]} Z)\cong\bbA^1_{\bbD}(J)(Y,Z\lhd_{[\bbA^1,I]} X)$$

The first isomorphism above is given by the composition

\begin{eqnarray*}
\bbA^1_{\bbD}(I\times J)(X\boxtimes_{\bbA^1_{\bbD}} Y,Z)&\cong&\bbD(\underline{\bbN}\times I\times J)((+\times 1)_!(X\boxtimes_{\bbD} Y),Z)\\
&\cong& \bbD(\underline{\bbN}\times\underline{\bbN}\times I\times J)(X\boxtimes_{\bbD} Y,(+\times 1)^*Z)\\
&\cong &\bbD(\underline{\bbN}\times I)(X,Y\rhd_{[J]} (+\times 1)^*Z)\\
&\cong &\bbD(\underline{\bbN}\times I)(X,Y\rhd_{[\bbA^1,J]} Z)\\
&\cong &\bbA^1_{\bbD}(I)(X,Y\rhd_{[\bbA^1,J]} Z)
\end{eqnarray*}

Here each step uses either the $(+_!,+^*)$-adjunction or $\boxtimes_{\bbD}$ being an adjunction of two variables.

Meanwhile the isomorphism $\bbA^1_{\bbD}(I\times J)(X\boxtimes_{\bbA^1_{\bbD}} Y,Z)\cong\bbA^1_{\bbD}(J)(Y,Z\lhd_{[\bbA^1,I]} X)$ is similarly given by 

\begin{eqnarray*}
\bbA^1_{\bbD}(I\times J)(X\boxtimes_{\bbA^1_{\bbD}} Y,Z)&\cong&\bbD(\underline{\bbN}\times I\times J)((+\times 1)_!(X\boxtimes_{\bbD} Y),Z)\\
&\cong& \bbD(\underline{\bbN}\times\underline{\bbN}\times I\times J)(X\boxtimes_{\bbD} Y,(+\times 1)^*Z)\\
&\cong&\bbD(\underline{\bbN}\times J)(Y,(+\times 1)^*Z\lhd_{[J]} X)\\
&\cong&\bbD(\underline{\bbN}\times J)(Y,Z \lhd_{[\bbA^1,J]} X)\\
&\cong&\bbA^1_{\bbD}(J)(Y,Z\lhd_{[\bbA^1,I]} X)
\end{eqnarray*}

Here again each isomorphism is either induced from the adjunction $(+_!,+^*)$ or the two-variable adjunction for $\bbD$, and is therefore natural.

For the canonical mates, fix $Y\in\bbA^1_{\bbD}(J)$, $Z\in\bbA^1_{\bbD}(I\times J)$, and let $u\colon I'\rightarrow I$ be any functor. Then we have 

\begin{eqnarray*}
u^*(Y\rhd_{[\bbA^1,J]} Z)&\cong& u^*(Y\rhd_{[\underline{\bbN}\times J]} +^*Z)\\
&\cong&(Y\rhd_{[\underline{\bbN}\times J]} (u\times1_{\underline{\bbN}\times\underline{\bbN}})^*+^*Z)\\
&\cong&(Y\rhd_{[\underline{\bbN}\times J]}(+^*(u\times 1_{\underline{\bbN}})^*Z)\\
&\cong&(Y\rhd_{[\bbA^1,J]}(u\times 1_{\underline{\bbN}})^*Z)\\
&\cong&(Y\rhd_{[\bbA^1,J]}(u\times 1)^*Z)
\end{eqnarray*}

where the isomorphisms come from naturality and the known natural isomorphisms for $\rhd_{[J]}$. Similarly for the other mate we fix $X\in\bbA^1_{\bbD}(I)$ and $Z\in\bbA^1_{\bbD}(J)$, and a functor $v\colon J'\rightarrow J$. Then we have

\begin{eqnarray*}
v^*(Z\lhd_{[\bbA^1,J]} X)&\cong&v^*(+^*Z\lhd_{[\underline{\bbN}\times J]} X)\\
&\cong&((1\times v)^*+^*Z\lhd_{[\underline{\bbN}\times J]} X)\\
&\cong& (+^*(1\times v)^*Z\lhd_{[\underline{\bbN}\times J]} X)\\
&\cong& ((1\times v)^*Z\lhd_{[\bbA^1,J]} X)\\
\end{eqnarray*}

This completes the verification that $\otimes_{\bbA^1}\colon\bbA^1_{\bbD}\times\bbA^1_{\bbD}\rightarrow\bbA^1_{\bbD}$ is a two-variable left adjoint, hence $(\bbA^1_{\bbD},\boxtimes_{\bbA^1_{\bbD}})$ is closed if $(\bbD,\boxtimes)$ is a closed monoidal derivator. 
\end{proof}

If $(\bbD,\boxtimes,\mathbbm{1})$ is a monoidal derivator, it comes equipped with a so-called ``projection morphism'': for objects $A\in\bbD(I)$, $B\in\bbD(J)$, and functor $u\colon I\rightarrow J$, we have a morphism $u_!(A\otimes_{\bbD(I)} u^*B)\rightarrow u_!A\otimes_{\bbD(J)} B$ defined to be the adjoint of the following morphism under the $u_!\dashv u^*$ adjunction.

$$\begin{CD}
A\otimes_{\bbD(I)} u^* B@>>>u^*u_!A_{\bbD(I)}\otimes u^*B\cong u^*(u_!A\otimes_{\bbD(J)} B)\\
\end{CD}$$

The projection morphism is moreover an isomorphism if $(\bbD,\boxtimes,\mathbbm{1})$ is a closed monoidal derivator, see \cite[Lemma 2]{Gallauer14}. Thus, $\bbA^1_{\bbD}$ is equipped with a projection isomorphism whenever $\bbD$ is closed.

\section{The universal property of $\bbA^1$}

Recall that our definition for the affine line over a derivator rested upon the intuition that if $R$ is a commutative ring, an $R[t]$-module is nothing more than an $R$-module with an $R$-module endomorphism, allowing us to define $\bbA^1_{\bbD}=\bbD^{\underline{\bbN}}$ for any derivator $\bbD$. For rings, a homomorphism $R[t]\rightarrow S$ can be broken down into a two simple parts, the ``underlying morphism'' $R\rightarrow S$, and the value of $t\in R[t]$ under the homomorphism. We expect a similar result for derivators, that a (cocontinuous) strong monoidal morphism $F\colon\bbA^1_{\bbD}\rightarrow\mathbb{E}$ for monoidal derivators $\bbD$, $\mathbb{E}$ can be determined by an ``underlying morphism'' $\bbD\rightarrow\mathbb{E}$ and the value of an object in $\mathbb{E}(\underline{\bbN})$.

\begin{lemma}
Let $F:\bbD\rightarrow\mathbb{E}$ be a cocontinuous strong monoidal morphism between monoidal derivators. Then $F^{\underline{\bbN}}\colon\bbA^1_{\bbD}\rightarrow\bbA^1_{\mathbb{E}}$ is again strong monoidal for the $\bbA^1_{\bbD}$, $\bbA^1_{\mathbb{E}}$-monoidal structures on $\mathbb{D}^{\underline{\bbN}}$, $\mathbb{E}^{\underline{\bbN}}$, respectively.
\end{lemma}

\begin{proof}
The following diagram commutes up to natural isomorphism: 

\centerline{
\xymatrix{
\bbD^{\underline{\bbN}}\times\bbD^{\underline{\bbN}} \ar[r]^{F^{\underline{\bbN}}\times F^{\underline{\bbN}}} \ar[d]_{\boxtimes_{\bbD}} &
\mathbb{E}^{\underline{\bbN}}\times\mathbb{E}^{\underline{\bbN}} \ar[d]^{\boxtimes_{\mathbb{E}}}\ar@{}[dl]|\swtrans \ar@{}[dl]<-1.5ex>|{\cong}\\
\bbD^{\underline{\bbN}\times\underline{\bbN}} \ar[d]_{+_!}\ar[r]_{F^{\underline{\bbN}\times\underline{\bbN}}}&
\mathbb{E}^{\underline{\bbN}\ar[d]_{+_!}\times\underline{\bbN}}\ar@{}[dl]|\swtrans \ar@{}[dl]<+1.5ex>|{\cong}\\
\bbD^{\underline{\bbN}}\ar[r]_{F^{\underline{\bbN}}}& \mathbb{E}^{\underline{\bbN}}}}

Here the commutativity of the top square up to natural isomorphism expresses that $F$ is strong monoidal, while the commutativity of the bottom square is a consequence of $F$ being cocontinuous.
\end{proof}

Note that in this case $F^{\underline{\bbN}}$ is again a cocontinuous strong monoidal functor, so we can iterate this construction.

Now we take some steps towards the decomposition discussed above. If we are given a morphism of derivators $F\colon\bbA^1_{\bbD}\rightarrow\mathbb{E}$, composing with $i_!\colon\bbD\rightarrow\bbD^{\underline{\bbN}}$ gives us the ``underlying'' morphism of derivator $F_0\colon\bbD\rightarrow\mathbb{E}$. Indeed this is what we would expect if $\bbD$ and $\mathbb{E}$ are derivators associated to rings.

\begin{lemma}
Let $F_0\colon\bbD\rightarrow\mathbb{E}$ be a cocontinuous strong monoidal morphism of derivators and $(\mathbbm{1},\alpha)\in\mathbb{E}(\underline{\bbN})$ be any coherent endomorphism. Then the composition $\textrm{ev}_{\alpha}\circ F_0^{\underline{\bbN}}\colon\bbA^1_{\bbD}\rightarrow\mathbb{E}$ is also a cocontinuous, strong monoidal morphism.
\end{lemma}

\begin{proof}
This is clear, since both $\textrm{ev}_{\alpha}$ and $F_0^{\underline{\bbN}}$ were known to be strong monoidal and cocontinuous.
\end{proof}

Moreover, such a morphism $\textrm{ev}_{\alpha}\circ F_0^{\underline{\bbN}}$ has underlying morphism $F_0$, since $\textrm{ev}_{\alpha}\circ F_0^{\underline{\bbN}}\circ i_!\cong \textrm{ev}_{\alpha}\circ i_!\circ F_0$ by co-continuity of $F_0$, and then by noting that $\textrm{ev}_{\alpha}\circ i_!$ is just the identity morphism.

\begin{theorem}
Let $\bbD$, $\mathbb{E}$ be two monoidal derivators, and let $F\colon\bbA^1_{\bbD}\rightarrow\mathbb{E}$ be a cocontinuous monoidal morphism. 

\begin{enumerate}
\item For a monoidal morphism of derivators $F:\bbA^1_{\bbD}\rightarrow\mathbb{E}$, $i^*F^{\underline{\bbN}}(+^*i_!\mathbbm{1}_{\bbD})=\mathbbm{1}_{\mathbb{E}}$. Therefore, $F^{\underline{\bbN}}(+^*i_!\mathbbm{1}_{\bbD})$ is a coherent endomorphism of the identity, say $(\mathbbm{1}_{\mathbb{E}},\alpha)$. Call $\alpha$ the \emph{type} of $F$. 

Let $F_0=F\circ i_!\colon\bbD\rightarrow\mathbb{E}$ be called the \emph{base} of the morphism $F$.

\item The only cocontinuous monoidal morphism with base $F_0$ and type $\alpha$ is $\textrm{ev}_{\alpha}\circ F_0^{\underline{\bbN}}$. That is to say, every cocontinuous monoidal morphism of derivators $\bbA^1_{\bbD}\rightarrow\mathbb{E}$ can be obtained as a composition $\textrm{ev}_{\alpha}\circ F_0^{\underline{\bbN}}$ for some coherent endomorphism $(\mathbbm{1}_{\mathbb{E}},\alpha)$ and some strong monoidal morphism $F_0:\bbD\rightarrow\mathbb{E}$.
\end{enumerate}
\label{maintheorem1}
\end{theorem}

First we see where the information of the coherent endomorphism $\alpha$ can be obtained.

\begin{lemma}
The following diagram commutes.

$$\begin{CD}
\bbD^{\underline{\bbN}\times\underline{\bbN}}@>F^{\underline{\bbN}}>>\mathbb{E}^{\underline{\bbN}}\\
@A +^*AA @V i^*VV\\
\bbD^{\underline{\bbN}}@>F>>\mathbb{E}
\end{CD}$$
\end{lemma}

\begin{proof}
Since the composition

$$\begin{CD}
\underline{\bbN}@>1\times i>>\underline{\bbN}\times\underline{\bbN}@>+>>\underline{\bbN}
\end{CD}$$

is the identity, so too is the composition $(1\times i)^*+^*$. Therefore, the below diagram commutes: 

$$\begin{CD}
\bbD^{\underline{\bbN}\times\underline{\bbN}}@>\textrm{Id}>>\bbD^{\underline{\bbN}\times\underline{\bbN}}\\
@A+^*AA @V(1\times i)^*VV\\
\bbD^{\underline{\bbN}}@>\textrm{Id}>>\bbD^{\underline{\bbN}}
\end{CD}$$

Then, for any morphism $\bbD^{\underline{\bbN}}\rightarrow\mathbb{E}$, we have a commutative diagram

$$\begin{CD}
\bbD^{\underline{\bbN}\times\underline{\bbN}}@>F^{\underline{\bbN}}>>\mathbb{E}^{\underline{\bbN}}\\
@V (1\times i)^*VV @V i^*VV\\
\bbD^{\underline{\bbN}}@>F>> \mathbb{E}
\end{CD}$$ Pasting the two squares horizontally gives precisely our desired diagram.
\end{proof}

In the specific case of $F^{\underline{\bbN}}(+^*\mathbbm{1}_{\bbA^1_{\bbD}})$, from being monoidal we know that $i^*F^{\underline{\bbN}}(+^*\mathbbm{1}_{\bbA^1_{\bbD}})$ is just $\mathbbm{1}_{\mathbb{E}}$. Therefore, $F^{\underline{\bbN}}(+^*\mathbbm{1}_{\bbA^1_{\bbD}})$ is equal to $(\mathbbm{1}_{\mathbb{E}},\alpha)\in\mathbb{E}(\underline{\bbN})$ for some $\alpha$. This endomorphism $\alpha$ is an important piece of information that we refer to as the ``type'' of our morphism $F$.

So let $F\colon\bbA^1_{\bbD}\rightarrow\mathbb{E}$ be a morphism of derivators. We wish to determine the image of $F^{\underline{\bbN}}(+^*X)$ for any $X\in\bbD^{\underline{\bbN}}$. First we give a decomposition of $+^*X$ as the $\bbA^1_{\bbA^1_{\bbD}}$-tensor product.

\begin{proposition}
$+^*X$ is the $\bbA^1_{\bbA^1_{\bbD}}$-tensor product of $(1\times i)_!X$ and $+^*i_!\mathbbm{1}_{\bbD}$. 
\end{proposition}

\begin{proof}
From $\eqref{A1monoid}$, recall the definition of the $\bbA^1$-monoidal structure. Here we end up taking a $\bbA^1_{\bbA^1_{\bbD}}$-tensor product, so there are two layers of complication.

For us, we consider $\mathbbm{1}_{\bbD}\in\bbD(e)$ and $X\in\bbA^1_{\bbD}(I)=\bbD(\underline{\bbN}\times I)$.

First we re-write $+^*X=+^*(X\boxtimes_{\bbA^1_{\bbD}} i_!\mathbbm{1}_{\bbD})$. Here we can draw upon the following diagram, expressing the definition of the $\bbA^1$-monoidal structure.

\centerline{
\xymatrix{
\bbD(\underline{\bbN})\times\bbD(\underline{\bbN}) \ar[r]^{\boxtimes_{\bbD}} \ar[d]_{\textrm{Id}} &
\bbD(\underline{\bbN}\times\underline{\bbN}) \ar[d]^{+_!}\ar@{}[dl]|\swtrans \ar@{}[dl]<-1.5ex>|{\textrm{Id}}\\
\bbA^1_{\bbD}(e)\times\bbA^1_{\bbD}(e)\ar[r]_{\boxtimes_{\bbA^1_{\bbD}}}&
\bbD(\underline{\bbN})\cong\bbA^1_{\bbD}(e)}}
So we can write $+^*(X\boxtimes_{\bbA^1_{\bbD}} i_!\mathbbm{1}_{\bbD})=+^*+_!(X\boxtimes_{\bbD} i_!\mathbbm{1}_{\bbD})$. Now, we also have

 \begin{eqnarray*}
(1\times i)_!X\boxtimes_{\bbA^1_{\bbA^1_{\bbD}}} +^*i_!\mathbbm{1}_{\bbD}&=&+_!((1\times i)_! X\boxtimes_{\bbA^1_{\bbD}} +^*i_!\mathbbm{1}_{\bbD})\\
&\cong&+_!(1\times i)_!(X\boxtimes_{\bbA^1_{\bbD}}+^*i_!\mathbbm{1}_{\bbD})\\
&\cong& X\boxtimes_{\bbA^1_{\bbD}}(+^*i_!\mathbbm{1}_{\bbD})\\
&=&(+\times 1)_!(X\boxtimes +^*i_!\mathbbm{1}_{\bbD})\\
&\cong&(+\times 1)_!(1\times +)^*(X\boxtimes i_!\mathbbm{1}_{\bbD})
 \end{eqnarray*} 
 
The individual isomorphisms are as follows. The first equality is just the definition of the $\bbA^1_{\bbA^1_{\bbD}}$-external product relative to the $\bbA^1_{\bbD}$-external product. The second isomorphism is from co-continuity of the external product. Therefore, $((1\times i)_!X\boxtimes_{\bbA^1_{\bbD}}+^*i_!\mathbbm{1}_{\bbD})\cong (i\times 1)_!(X\boxtimes_{\bbA^1_{\bbD}} +^*i_!\mathbbm{1}_{\bbD})$, 
 
For the third isomorphism, having been left with $+_!(1\times i)_!$, by naturality we know that $+_!(1\times i)_!\cong (+\circ (1\times i))_!$, but $+\circ (1\times i)$ is just the identity functor on $\underline{\bbN}$. Thus, $+_! (1\times i)_!\cong 1_!$, and $1_!\cong\textrm{Id}$. Hence we can simply remove $+_!(1\times i)_!$ for the third isomorphism.

The fourth equality is once again a definition of the $\bbA^1_{\bbD}$-external product relative to the $\bbD$-external product, while the fifth isomorphism is a reflection of the fact that taking the external product with any object is a morphism of derivators. 

Thus, we have $(1\times i)_!X\boxtimes_{\bbA^1_{\bbA^1_{\bbD}}} +^*i_!\mathbbm{1}_{\bbD}=(+\times 1)_!(1\times +)^*(X\boxtimes_{\bbD} i_!\mathbbm{1}_{\bbD})$. We would like to show it to be isomorphic to $+^*(M,f)=+^*+_!(i_!\mathbbm{1}_{\bbD}\boxtimes_{\bbD} (M,f))$. A sufficient statement would be simply that $(+\times 1)_!(1\times +)^*\cong +^*+_!$. 
\end{proof}
 
\begin{lemma}
The following square is homotopy exact:
 
\begin{equation}
\centerline{
\xymatrix{
\underline{\bbN}\times\underline{\bbN}\times\underline{\bbN} \ar[r]^{1\times +} \ar[d]_{1\times +} &
\underline{\bbN}\times\underline{\bbN}\ar[d]^{+}\ar@{}[dl]|\swtrans \ar@{}[dl]<-1.5ex>|{\textrm{Id}}\\
\underline{\bbN}\times\underline{\bbN} \ar[r]_{+}&
\underline{\bbN} }}
\label{square1}%
\end{equation}

Here the natural transformation in the middle is just the identity, as both compositions are just $$\underline{\bbN}\times\underline{\bbN}\times\underline{\bbN}\rightarrow\underline{\bbN},$$ taking a map $(a,b,c)\mapsto a+b+c$.
\end{lemma}

\begin{proof}
Using (Der4), we know that the square 

\begin{equation}
\centerline{
\label{square2}
\xymatrix{
(+\times 1)/e\ar[r]^{pr} \ar[d]_{\pi} &
\underline{\bbN}\times\underline{\bbN}\times\underline{\bbN} \ar[d]^{+\times 1}\ar@{}[dl]|\swtrans \ar@{}[dl]<-1.5ex>|{\alpha}\\
e \ar[r]_{i}&
\underline{\bbN}\times\underline{\bbN} }}
\end{equation}
is homotopy exact. Our original square is homotopy exact if and only if its pasting with the above is homotopy exact, as homotopy exactness can be checked pointwise and $\underline{\bbN}\times\underline{\bbN}$ has precisely one object.

The category $(+\times 1/e)$, by definition has objects $(\bullet\in \underline{\bbN}\times\underline{\bbN}, (a,b)\colon\bullet\rightarrow\bullet)$, of which the information we can just condense to $(a,b)\in\bbN\times\bbN$. A morphism $(a,b)\rightarrow (c,d)$ will be a morphism $(j,k,l)$ in $\underline{\bbN}\times\underline{\bbN}\times\underline{\bbN}$ such that $(+\times 1)(j,k,l)=(a-c,b-d)$-i.e. that $j+k+c=a$ and $l+d=b$.

Hence, if the pasting of $\eqref{square1}$ and $\eqref{square2}$ can be shown to be homotopy exact, from the homotopy exactness of $\eqref{square2}$ we would obtain that $\eqref{square1}$ is homotopy exact. This pasting of $\eqref{square1}$ and $\eqref{square2}$ looks like

\begin{equation}
\label{pasting}
\centerline{
\xymatrix{
(+\times 1)/e \ar[r]^{pr} \ar[d]_{\pi} &
\underline{\bbN}\times\underline{\bbN}\times\underline{\bbN} \ar[r]^{1\times +}\ar[d]_{+\times 1}\ar@{}[dl]|\swtrans \ar@{}[dl]<-1.5ex>|{\alpha}&
\underline{\bbN}\times\underline{\bbN} \ar[d]^{+} \ar@{}[dl]|\swtrans \ar@{}[dl]<-1.5ex>|{\textrm{Id}}\\
e \ar[r]_{i}& \underline{\bbN}\times\underline{\bbN}\ar[r]^{+} &\underline{\bbN}} }
\end{equation}

Then we can whisker the natural transformations, to make this a single square with natural transformation. Below, we take the functor $p$ to be the composition of the top line in \eqref{pasting}: $$(+\times 1)/e\rightarrow\underline{\bbN}\times\underline{\bbN}\times\underline{\bbN}\rightarrow\underline{\bbN}\times\underline{\bbN}$$

\begin{equation}
\centerline{
\xymatrix{
(+\times 1)/e \ar[r]^{p} \ar[d]_{\pi}&
\underline{\bbN}\times\underline{\bbN} \ar[d]^{+}\ar@{}[dl]|\swtrans \ar@{}[dl]<-1.5ex>|{\alpha}\\
e \ar[r]_{i}&
\underline{\bbN}} }
\label{square3}
\end{equation}

It is not clear why this square would be homotopy exact given its current description. From (Der4) we consider the square

\begin{equation}
\centerline{
\xymatrix{
(+/e) \ar[r]^{pr} \ar[d]_{\pi} &
\underline{\bbN}\times\underline{\bbN} \ar[d]^{+}\ar@{}[dl]|\swtrans \ar@{}[dl]<-1.5ex>|{\alpha}\\
e \ar[r]_{i}&
\underline{\bbN} }} 
\label{square4}
\end{equation}

which we know to be homotopy exact. Here an object of the category $(+/e)$ has objects $(\bullet\in\underline{\bbN},m\colon\bullet\rightarrow\bullet)$, information that we can just condense to $m\in\bbN$. A morphism $(m)\rightarrow (n)$ is a morphism $(i,j)$ in $\underline{\bbN}\times\underline{\bbN}$ such that $+(i,j))=m-n$. We would like to write $\eqref{square3}$ as a pasting of $\eqref{square4}$ with another square, i.e. obtain a pasting of the form 

\begin{equation}
\centerline{
\xymatrix{
(+\times 1)/e \ar[r]^{G} \ar[d]_{\pi}&
(+/e)\ar[r]^{pr}\ar[d]_{\pi}\ar@{}[dl]|\swtrans \ar@{}[dl]<-1.5ex>|{\alpha}&
\underline{\bbN}\times\underline{\bbN} \ar[d]^{+}\ar@{}[dl]|\swtrans \ar@{}[dl]<-1.5ex>|{\textrm{Id}}\\
e \ar[r]_{Id}&
e\ar[r]_{i} &\underline{\bbN}} }
\label{square5}
\end{equation} such that $\textrm{pr}_{(+/e)}\circ G=(1\times +)\circ\textrm{pr}_{(+\times 1)/e}$, \ie we want to find a functor $G:(+\times 1/e)\rightarrow (+/e)$ such that the below square commutes:

$$\begin{CD}
(+\times 1/e)@>G>> (+/e)\\
@V\emph{pr}VV @V\emph{pr}VV\\
\underline{\bbN}\times\underline{\bbN}\times\underline{\bbN}@>1\times +>>\underline{\bbN}\times\underline{\bbN}
\end{CD}$$

Upon further examination, such a functor will be induced by $(1\times +)$ in the following way; taking an object $(a,b)$ in $(+\times 1/e)$ to $(a+b)\in (+/e)$, and a morphism $(j,k,l)\colon (a,b)\rightarrow (c,d)$ to $(j,k+l)\colon (a+b)\rightarrow (c+d)$. Checking that the above commutes tells us that $G$ is precisely what is required. Thus, we can re-write $\eqref{square3}$ in the guise of $\eqref{square5}$. 

The right-hand square of this pasting $\eqref{square5}$ is homotopy exact by (Der4), so it simply suffices to prove the left-hand square is homotopy exact. This most obvious step would be to check that the the functor $G$ is a right adjoint, by \cite[Proposition 1.24]{Groth13}, but this fails. Instead, let us denote $(+\times 1/e)=\cat C$, $(+/e)=\cat D$, and $\cat C_0\subset\cat C$ be the full subcategory with objects $(a,0)$. We form the pasted square 

\begin{equation}
\centerline{
\xymatrix{
\cat C_0 \ar[r]^{i_0} \ar[d]_{\pi}&
\cat C \ar[d]^{\pi}\ar[r]^{G}\ar@{}[dl]|\swtrans \ar@{}[dl]<-1.5ex>|{\textrm{Id}}&
\cat D\ar[d]^{\pi}\ar@{}[dl]|\swtrans \ar@{}[dl]<-1.5ex>|{\textrm{Id}}\\
e \ar[r]_{1}&
e\ar[r]_{1}& e }}
\label{square6}
\end{equation}

Here, I claim that both the inclusion $i_0:\cat C_0\hookrightarrow\cat C$ and $G\circ i_0$ are right adjoints. This will prove that both the left-hand square and the pasting are homotopy exact squares, and hence that the right-hand square is. We detail the respective adjunctions.

The left adjoint $L$ to $i_0:\cat C_0\hookrightarrow\cat C$ takes $(a,b)\in\cat C$ to $(a,0)\in\cat C_0$ and a map $(j,k,l)\colon (a,b)\rightarrow (c,d)$ to $(j,k,0):(a,0)\rightarrow (c,0)$. So we take $(a,b)\in\cat C$ and $(c,0)\in\cat C_0$, then $\textrm{Hom}_{\cat C_0}((a,0),(c,0))\cong\textrm{Hom}_{\cat C}((a,b),(c,0))$. The former consists of maps of the form $(j,k,0)$ with $j+k+c=a$, while the latter consists of maps of the form $(j,k,b)$ with $j+k+c=a$, rendering an obvious bijection.

The left adjoint $F$ to $\cat C_0\hookrightarrow\cat C\rightarrow\cat D$ takes $(m)\in\cat D$ to $(m,0)\in\cat C_0$ and a map $(a,b):(m)\rightarrow (n)$ to $(a,b,0):(a,0)\rightarrow (c,0)$. Take $(a,0)\in\cat C_0$ and $(n)\in\cat D$. Then $\textrm{Hom}_{\cat C_0}((n,0),(a,0))\cong\textrm{Hom}_{\cat D}((n),(a))$, as the former consists of maps $(j,k,0)$ where $j+k+a=n$, while the latter consists of maps $(j,k)$ where $j+k+a=n$, with the obvious bijection. Therefore, both $i_0:\cat C_0\hookrightarrow \cat C$ and $G\circ i_0$ are both right adjoints.

Therefore, the square

\begin{equation}
\centerline{
\xymatrix{
\cat C \ar[r]^{G} \ar[d]_{\pi} &
\cat D \ar[d]^{\pi} \ar@{}[dl]|\swtrans \ar@{}[dl]<-1.5ex>|{\alpha}\\
e \ar[r]_{1}&
e }}
\end{equation}

is homotopy exact. This implies that $\eqref{square5}$ is homotopy exact, which is the same square as $\eqref{square3}$. Therefore, our original square $\eqref{square1}$ is homotopy exact. Recall it is the below square:

\begin{equation}
\centerline{
\xymatrix{
\underline{\bbN}\times\underline{\bbN}\times\underline{\bbN} \ar[r]^{1\times +} \ar[d]_{+\times 1}&
\underline{\bbN}\times\underline{\bbN} \ar[d]^{+}\ar@{}[dl]|\swtrans \ar@{}[dl]<-1.5ex>|{\textrm{Id}}\\
\underline{\bbN}\times\underline{\bbN} \ar[r]_{+}&
\underline{\bbN} }}
\end{equation}

Taking the respective adjoints, we have an isomorphism $$(+\times 1)_!(1\times +)^*\cong +^*+_!.$$ 
\end{proof}

Now we can complete the proof of the proposition.

\begin{proof}
From the previous lemma we have 

\begin{eqnarray*}
+^*X &=& +^* (i_!\mathbbm{1}_{\bbD}\boxtimes_{\bbA^1_{\bbD}} X)\\
&=& +^* +_!(i_!\mathbbm{1}_{\bbD}\boxtimes_{\bbD} X)
\end{eqnarray*}

As $+^*+_!\cong (+\times 1)_!(1\times +)^*$, we have $$+^* +_!(i_!\mathbbm{1}_{\bbD}\boxtimes_{\bbD} X)\cong(+\times 1)_!(1\times +)^*(i_!\mathbbm{1}_{\bbD}\boxtimes_{\bbD} X).$$ We also have that $$(1\times i)_!X\boxtimes_{\bbA^1_{\bbA^1_{\bbD}}} +^*i_!\mathbbm{1}_{\bbD}=(+\times 1)_!(1\times +)^*(X\boxtimes_{\bbD} i_!\mathbbm{1}_{\bbD}).$$ The proof of the theorem will rely on the two isomorphic representations of $+^*X$ that we have produced.
\end{proof}

Lastly, we can tackle the proof of the Theorem. Roughly speaking, for a cocontinuous monoidal morphism of derivators $F\colon\bbA^1_{\bbD}\rightarrow\mathbb{E}$, $F_0=Fi_!$ and $F^{\underline{\bbN}}(+^*i_!\mathbbm{1}_{\bbD})=(\mathbbm{1}_{\mathbb{E}},\alpha)$ correspond to $R\rightarrow S$ and the assignment $t\in R[t]$ respectively.

\begin{proof} 
We know that $F^{\underline{\bbN}}((1\times i)_!X)=F_0^{\underline{\bbN}}X$, while $F^{\underline{\bbN}}(+^*i_!\mathbbm{1}_{\bbD})=(\mathbbm{1}_{\mathbb{E}},\alpha)$ by characterization of having type $\alpha$. Therefore, $F^{\underline{\bbN}}(+^*X)=F_0^{\underline{\bbN}}X\boxtimes_{\bbA^1_{\bbE}} (\mathbbm{1}_{\mathbb{E}},\alpha)$. Thus,
 
 \begin{eqnarray*}
 F(X)&\cong&i^*F^{\underline{\bbN}}(+^*X)\\
 &\cong&i^*F^{\underline{\bbN}}(i_!X\boxtimes_{\bbA^1_{\bbA^1_{\bbD}}}(+^*(i_!\mathbbm{1}_{\bbD})))\\
 &\cong&i^*(F^{\underline{\bbN}}(i_!X)\boxtimes_{\bbA^1_{\mathbb{E}}}F^{\underline{\bbN}}(+^*\mathbbm{1}_{\bbA^1_{\bbD}}))\\
 &=&i^*(F_0^{\underline{\bbN}}X\boxtimes_{{\bbA^1}_{\mathbb{E}}}(\mathbbm{1}_{\mathbb{E}},\alpha))\\
 &=&\textrm{ev}_{\alpha}F_0^{\underline{\bbN}}X
 \end{eqnarray*}
 
 Above, the first isomorphism $F(X)\cong i^*F^{\underline{\bbN}}(+^*X)$ is since $F$ is a morphism of derivators. The second isomorphism is the decomposition $$+^*X=+^*i_!\mathbbm{1}_{\bbD}\boxtimes_{\bbA^1_{\bbA^1_{\bbD}}}(1\times i)_!X,$$ and the third just follows by co-continuity of $F$. 
 
Our very last equality is precisely the definition of the evaluation at $\alpha$ morphism.
 
Therefore, cocontinuous monoidal morphisms of derivators $F:\bbA^1_{\bbD}\rightarrow\mathbb{E}$ can be determined simply by their base $F_0$ and their type $\alpha$, in that a morphism $F$ with designated base $F_0$ and type $\alpha$ is simply $\textrm{ev}_{\alpha} F_0^{\underline{\bbN}}$.
\end{proof}

\begin{definition}
Let $\bbD$, $\mathbb{E}$ be two derivators. Let $PDER(\bbD,\mathbb{E})$ denote the category whose objects are morphisms of prederivators $\bbD\rightarrow\mathbb{E}$ and morphisms are modifications. 

Let $PDER_!(\bbD,\mathbb{E})$ denote the full category whose objects are cocontinuous morphisms of derivators $\bbD\rightarrow\mathbb{E}$]. Similarly, if $\bbD$ and $\mathbb{E}$ are monoidal derivators, then we can let $PDER_{\otimes}(\bbD,\mathbb{E})$ denote the category of monoidal morphisms and pseudonatural transformations. Finally, if we are looking at $PDER(\bbA^1_{\bbD},\mathbb{E})$ we can let $PDER_{\otimes,!,\bbD}(\bbA^1_{\bbD},\mathbb{E})$ denote the cocontinuous, monoidal morphisms $F:\bbA^1_{\bbD}\rightarrow\mathbb{E}$ such that $F\circ i_!$ is a given morphism $\bbD\rightarrow\mathbb{E}$. 
\end{definition}

\begin{theorem}
We describe $\eqref{maintheorem1}$ via a more global perspective.

\begin{enumerate}
\item There exists a functor $$PDER_{\otimes,!}(\bbA^1_{\bbD},\mathbb{E})\rightarrow\mathbb{E}(\underline{\bbN})$$ that sends a cocontinuous, monoidal morphism $F\colon\bbA^1_{\bbD}\rightarrow\mathbb{E}$ to $F(+^*i_!\mathbbm{1}_{\bbD})$. 

\item Fix a cocontinuous monoidal morphism of derivators $F_0\colon\bbD\rightarrow\mathbb{E}$. If we restrict to cocontinuous monoidal morphisms with base $F_0$, this induces an equivalence of categories $$\textrm{PDER}_{\otimes,!,\bbD}(\bbA^1_{\bbD},\mathbb{E})\cong\{X:i^*X=\mathbbm{1}_{\mathbb{E}}\}\subset\bbA^1_{\mathbb{E}}(e).$$ Call the latter category $\mathbb{E}_{\mathbbm{1}}$, which we can also think of as all the coherent endomorphisms of the identity in $\mathbb{E}$. Here $PDER_{\otimes,!,\mathbb{D}}(\mathbb{A}^1_{\bbD},\mathbb{E})$ consists of cocontinuous, strong monoidal morphisms $F$ of derivators between $\bbA^1_{\bbD}\rightarrow\mathbb{E}$ with $Fi_!$ equal to some fixed $F_0$. 

In the above equivalence, the forward direction functor $$\textrm{PDER}_{\otimes,!,\bbD}(\bbA^1_{\bbD},\mathbb{E})\rightarrow\mathbb{E}_{\mathbbm{1}}$$ is the functor in the first part. Its inverse takes $(\mathbbm{1},\alpha)$ to $\textrm{ev}_{\alpha}\circ F_0^{\underline{\bbN}}$. 

\item Alternatively, we have an equivalence of categories $$\textrm{PDER}_{\otimes,!}(\bbA^1_{\bbD},\mathbb{E})\cong\mathbb{E}_{\mathbbm{1}}\times\textrm{PDER}_{\otimes,!}(\bbD,\mathbb{E}).$$ The forward direction functor splits $F$ into the information of its type $\alpha$ and its base $F_0$. Its inverse takes a coherent endomorphism $(\mathbbm{1}_{\mathbb{E}},\alpha)$ plus a cocontinuous, monoidal morphism $F_0\colon\bbD\rightarrow\mathbb{E}$ to $\textrm{ev}_{\alpha}F_0^{\underline{\bbN}}$. 
\end{enumerate}
\label{maintheorem2}
\end{theorem}

\begin{proof}
\begin{enumerate}
\item The first part is clear, every pseudonatural transformation between two morphisms $F,G:\bbA^1_{\bbD}\rightarrow\mathbb{E}$ gives a morphism in $\mathbb{E}(\underline{\bbN})$, $$F(+^*i_!\mathbbm{1}_{\bbD})\rightarrow G(+^*i_!\mathbbm{1}_{\bbD}).$$

\item First it is clear by $\eqref{maintheorem1}$ that the functor in part (1) can have its codomain restricted to to $\mathbb{E}_{\mathbbm{1}}$. 

The two functors between $PDER_{\otimes, !,\mathbb{D}}(\bbA^1_{\bbD},\mathbb{E})$ and $\mathbb{E}_{\mathbbm{1}}$ are as follows. Given $F:\bbA^1_{\bbD}\rightarrow\mathbb{E}$, we send it to $F_{\underline{\bbN}}(+^*i_!\mathbbm{1}_{\bbD})\in\mathbb{E}_{\mathbbm{1}}$.  For a monoidal natural transformation $F\rightarrow F'$, we send it to the morphism $$F^{\underline{\bbN}}(+^*i_!\mathbbm{1}_{\bbD})\rightarrow F'^{\underline{\bbN}}(+^*i_!\mathbbm{1}_{\bbD}).$$

Conversely, given $(\mathbbm{1},\alpha)\in\mathbb{E}_{\mathbbm{1}}$ we send it to $\textrm{ev}_{\alpha}\circ F_0^{\underline{\bbN}}\colon\bbA^1_{\bbD}\rightarrow\mathbb{E}$, which we know to be in $PDER_{\otimes,!,\mathbb{D}}(\mathbb{A}^1_{\bbD},\mathbb{E})$. Given a morphism $(\mathbbm{1},\alpha)\rightarrow (\mathbbm{1},\beta)\in\mathbb{E}_{\mathbbm{1}}$, there is the corresponding monoidal natural transformation $\textrm{ev}_{\alpha}\rightarrow\textrm{ev}_{\beta}$. Recall the definition of the evaluation at $\alpha$ morphism, $\eqref{ev alpha}$; given a morphism $g\colon(\mathbbm{1},\alpha)\rightarrow (\mathbbm{1},\beta)$ there is a corresponding morphism, $\textrm{Id}_X\boxtimes_{\bbA^1_{\bbD}} g$ from $$i^*(X\boxtimes_{\bbA^1_{\bbD}} (\mathbbm{1},\alpha))\rightarrow i^*(X\boxtimes_{\bbA^1_{\bbD}}(\mathbbm{1},\beta))$$ and these paste to become a monoidal natural transformation $\textrm{ev}_{\alpha}\rightarrow\textrm{ev}_{\beta}$. Hence we also obtain a monoidal natural transformation $\textrm{ev}_{\alpha} F_0^{\underline{\bbN}}\rightarrow\textrm{ev}_{\beta}F_0^{\underline{\bbN}}$.

The compositions $PDER_{\otimes,!,\mathbb{D}}(\mathbb{A}^1_{\bbD},\mathbb{E})\rightarrow\mathbb{E}_{\mathbbm{1}}\rightarrow PDER_{\otimes,!,\mathbb{D}}(\mathbb{A}^1_{\bbD},\mathbb{E})$ and $\mathbb{E}_{\mathbbm{1}}\rightarrow PDER_{\otimes,!,\mathbb{D}}(\mathbb{A}^1_{\bbD},\mathbb{E})\rightarrow\mathbb{E}_{\mathbbm{1}}$ are both seen to be isomorphisms, showing that there is actually an equivalence between the two categories.

\item As in the previous part, given $F\in PDER_{\otimes,!}(\bbA^1_{\bbD},\mathbb{E})$ we map it to $$(F(+^*i_!\mathbbm{1}_{\bbD}),F\circ i_!)$$ For a morphism $F\rightarrow G$ in $PDER_{\otimes,!}(\bbA^1_{\bbD},\mathbb{E})$, we can take them to the morphisms $F(+^*i_!\mathbbm{1}_{\bbD})\rightarrow G(+^*i_!\mathbbm{1}_{\bbD})$ and $F\circ i_!\rightarrow G\circ i_!$ in the categories $\mathbb{E}_{\mathbbm{1}}$ and $PDER_{\otimes,!}(\bbD,\mathbb{E})$. 

In the opposite direction, we take a pair $(\mathbbm{1}_{\mathbb{E}},\alpha)$ and $F_0\colon\bbD\rightarrow\mathbb{E}$ and send it to $$\textrm{ev}_{\alpha}\circ F_0^{\underline{\bbN}}\colon\bbA^1_{\bbD}\rightarrow\mathbb{E}.$$ Given a pair of morphisms $$(\mathbbm{1}_{\mathbb{E}},\alpha)\rightarrow (\mathbbm{1}_{\mathbb{E}},\beta), F_0\rightarrow G_0,$$ we have a transformation $$\textrm{ev}_{\alpha} F_0^{\underline{\bbN}}\rightarrow \textrm{ev}_{\alpha} G_0^{\underline{\bbN}}\rightarrow \textrm{ev}_{\beta} G_0^{\underline{\bbN}}$$ where the first arrow is given by the transformation $F_0\rightarrow G_0$ and the second induced by the transformation $\textrm{ev}_{\alpha}\rightarrow\textrm{ev}_{\beta}$ as described in the previous part. From $\eqref{maintheorem1}$, we know that these two functors are essential inverses to each other.
\end{enumerate}
\end{proof}

We note one special case below, when the base is the identity.

\begin{corollary}
The only monoidal morphisms $F\colon\bbA^1_{\bbD}\rightarrow\bbD$ that $F\circ i_!$ is the identity are the \emph{evaluation at $\alpha$} morphisms $\textrm{ev}_{\alpha}$. 
\end{corollary}

Indeed it's clear by the definition of the evaluation at $\alpha$ morphism that each one is a section to the structure morphism $i_!$. Now we have seen that these are the only sections. 

Now we aim to extend this result to $\bbA^n_{\bbD}$ in a natural way. Let us fix the following notation.

\begin{enumerate}
\item Let $i_n\colon\underline{\bbN}^{n-1}\rightarrow\underline{\bbN}^n$ denote the functor $1_{\underline{\bbN}^{n-1}}\times i$
\item Let $0\leq m<n$ be two integers. Let $i_{m,n}$ denote the composition $$i_n\circ i_{n-1}\circ\cdots\circ i_{m+1}\colon\underline{\bbN}^m\rightarrow\underline{\bbN}^n$$ 
\item Let $F\colon\bbA^n_{\bbD}\rightarrow\mathbb{E}$ be a cocontinuous monoidal morphism of derivators. Let $F_{m,n}$ denote the morphism $$F\circ (i_{m,n})_!\colon\bbA^m_{\bbD}\rightarrow\mathbb{E}$$ for $0\leq m<n$. 
\end{enumerate}

The idea for generating a similar universal property for $\bbA^n$ basically consists of iterating the universal property for $\bbA^1$ along the inclusions $i_k$. Note that if we think of $\bbA^n_{\bbD}$ as $\bbA^1_{\bbA^{n-1}_{\bbD}}$, then with the above notation $F_{n-1,n}$ is the base of the morphism of derivators $F\colon\bbA^n_{\bbD}\rightarrow\mathbb{E}$, and more generally $F_{m-1,n}$ is the base of the  morphism $$F_{m,n}\colon\bbA^m_{\bbD}\rightarrow\mathbb{E}.$$ So we can use the universal property of $\bbA^1$ (\eqref{maintheorem1}, \eqref{maintheorem2}) to bootleg up to $\bbA^n$.

\begin{corollary}
We have the following analogues of \eqref{maintheorem2}. 
\begin{enumerate}
\item There is a functor $$PDER_{\otimes,!}(\bbA^n_{\bbD},\mathbb{E})\rightarrow\Pi_{n}\mathbb{E}(\underline{\bbN}),$$ sending a cocontinuous monoidal morphism $F\cong\bbA^n_{\bbD}\rightarrow\mathbb{E}$ to the product of the types of $F_{m,n}$ for each $1\leq m\leq n$, from the universal property of $\bbA^1$.
\item Fix a cocontinuous, monoidal morphism of derivators $G\colon\bbD\rightarrow\mathbb{E}$. Consider the subcategory $$PDER_{\otimes,!,\bbD}(\bbA^n_{\bbD},\mathbb{E})\subset PDER_{\otimes,!}(\bbA^n_{\bbD},\mathbb{E})$$ consisting of morphisms $F$ with $F_{0,n}=G$. The functor in part (1) induces an equivalence of categories with $\Pi_{n}\mathbb{E}_1$. 
\item From (2) we have an equivalence of categories $PDER_{\otimes,1}(\bbA^n_{\bbD},\mathbb{E})$ with the category $$PDER_{\otimes,!}(\bbD,\mathbb{E})\times\Pi_n\mathbb{E}_1.$$
\end{enumerate}
\end{corollary}

\begin{proof}
First we will describe how to generate the product $\Pi_n\mathbb{E}_1$ by induction. The main theorem \eqref{maintheorem2} is the case for $n=1$. Inductively, note the commutative diagram

\centerline{
\xymatrix{
 &\bbA^1_{\mathbb{E}} \ar[d]^{\textrm{ev}_{\alpha_n}}\\
 \bbA^n_{\bbD}\ar[ur]^{F_{n-1,n}^{\underline{\bbN}}}\ar[r]_{F}&\mathbb{E}\\
 \bbA^{n-1}_{\bbD}\ar[u]^{(i_n)_!} \ar[ur]_{F_{n-1,n}} &
  }} where both triangles commute. The $\textrm{ev}_{\alpha_n}$ is obtained from the universal property of $\bbA^1$ on for the derivator $\bbA^{n-1}_{\bbD}$, which gives the commutativity of the top triangle. So for $\bbA^n$ we get the information of $n$ evaluation at $\alpha$ morphisms giving us the requisite functor for (1). 
 
 For (2) it's clear that one can restrict the codomain of the functor in (1) to $\Pi_n \mathbb{E}_1$. Note that suppose we are given a cocontinuous monoidal morphism $F\colon\bbA^n_{\bbD}\rightarrow\mathbb{E}$ with $F_{0,n}=G$. Then $F_{n-1,n}\colon\bbA^{n-1}_{\bbD}\rightarrow\mathbb{E}$ again has base $G$ and we can write $F=\textrm{ev}_{\alpha_n}\circ F_{n-1,n}^{\underline{\bbN}}$. Then one can write $F_{n-1,n}=\textrm{ev}_{\alpha_{n-1}}\circ F_{n-2,n}^{\underline{\bbN}}$, and so forth. 

So given a product $\Pi_{i=1}^{n} (\mathbbm{1}_{\mathbb{E}},\alpha_i)$ and a morphism $G$ which should be equal to $F_{0,n}$ for some $F\colon\bbA^n_{\bbD}\rightarrow\mathbb{E}$, we can recursively define $F_{k+1,n}$ as $\textrm{ev}_k\circ (F_{k,n})^{\underline{\bbN}}$, until we get to $F_{n,n}$ which is simply our desired morphism of derivators.

This gives the inverse to our stated equivalence in part (2), while (3) is simply a re-writing of (2).  
\end{proof}

For us, this is the most telling signal that the definition of $\bbA^1_{\bbD}$ and more generally $\bbA^n_{\bbD}$ is a reasonable one. It mirrors behavior that we would expect polynomial algebras over a commutative ring or the affine spaces over a reasonable scheme to have.

\section*{Acknowledgements}

I would like to thank my advisor, Paul Balmer, for his comments and advice throughout the writing of this paper. Special thanks to Ian Coley for his careful and very helpful edits.


\begin{thebibliography}{Gro13}

\bibitem[Bal10]{Balmer10}
Paul Balmer.
\newblock Tensor Triangular Geometry.
\newblock {\em Proceedings ICM, Hyderabad}, (2010), Vol. II, pp. 85-112.

\bibitem[Bal14]{Balmer14pp}
Paul Balmer.
\newblock The derived category of an \'etale extension and the separable
  {N}eeman-{T}homason theorem.
\newblock {\em J. Inst. Math. Jussieu}, 2014.
\newblock Preprint, 9~pages, to appear.

\bibitem[BN93]{BokstedtNeeman93}
Marcel B{\"o}ksedt and Amnon Neeman.
\newblock Homotopy limits in triangulated categories.
\newblock{Compositio Mathematica}, 1993.
Volume: 86, Issue: 2, page 209-234


\bibitem[BZ16]{BalmerZhang16}
Paul Balmer and John Zhang.
\newblock Affine space over triangulated categories:\\ A further invitation to Grothendieck derivators.
\newblock http://arxiv.org/pdf/1601.02576v1.pdf

\bibitem[Cis03]{Cisinski03}
Denis-Charles Cisinski.
\newblock Images directes cohomologiques dans les cat\'egories de mod\`eles.
\newblock {\em Ann. Math. Blaise Pascal}, 10(2):195--244, 2003.

\bibitem[CN08]{CisinskiNeeman08}
Denis-Charles Cisinski and Amnon Neeman.
\newblock Additivity for derivator {$K$}-theory.
\newblock {\em Adv. Math.}, 217(4):1381--1475, 2008.

\bibitem[Gal14]{Gallauer14}
Martin Gallauer Alves de Souza.
\newblock Traces in monoidal derivators, and homotopy colimits.
\newblock{\em Adv. Math.}, 261: 26-84, 2014.

\bibitem[Gro60]{Grothendieck60}
Alexander Grothendieck.
\newblock The cohomology theory of abstract algebraic varieties.
\newblock In {\em Proc. {I}nternat. {C}ongress {M}ath. ({E}dinburgh, 1958)},
  pages 103--118. Cambridge Univ. Press, New York, 1960.

\bibitem[Gro91]{Grothendieck91}
A.~Grothendieck.
\newblock Les d\'erivateurs, 1991.
\newblock Manuscript, transcripted by M. K\"unzer, J. Malgoire and G.
  Maltsiniotis and available online
  at~\url{http://webusers.imj-prg.fr/~georges.maltsiniotis/groth/Derivateurs.html}.
  
\bibitem[Gro11]{Groth11}
Moritz Groth.
\newblock Monoidal derivators and enriched derivators.
~\url{http://www.math.ru.nl/~mgroth/preprints/groth_enriched.pdf}.

\bibitem[Gro13]{Groth13}
Moritz Groth.
\newblock Derivators, pointed derivators and stable derivators.
\newblock {\em Algebr. Geom. Topol.}, 13(1):313--374, 2013.

\bibitem[Gro16]{Groth16}
Moritz Groth.
\newblock Book project on derivators, volume I.
\newblock Manuscript, available online 
at~\url{http://www.math.uni-bonn.de/~mgroth/monos/intro-to-der-1.pdf}.

\bibitem[GPS14a]{GrPoSh14a}
Moritz Groth, Kate Ponto, and Mike Shulman.
\newblock Mayer-Vietoris sequences in stable derivators.
\newblock{\em Homology, Homotopy, and applications}, 16(1):265-294, 2014. 

\bibitem[GPS14b]{GrPoSh14b}
Moritz Groth, Kate Ponto, and Mike Shulman.
\newblock The additivity of traces in monoidal derivators.
\newblock {\em Journal of K-Theory}, 14(3): 422-494, 2014.

\bibitem[Har66]{Hartshorne66}
Robin Hartshorne.
\newblock {\em Residues and duality}.
\newblock Lecture notes of a seminar on the work of A. Grothendieck, given at
  Harvard 1963/64. With an appendix by P. Deligne. Lecture Notes in
  Mathematics, No. 20. Springer, Berlin-New York, 1966.

\bibitem[Hel88]{Heller88}
Alex Heller.
\newblock Homotopy theories.
\newblock {\em Mem. Amer. Math. Soc.}, 71(383):vi+78, 1988.

\bibitem[MR14]{MuroRaptis14pp}
Fernando Muro and George Raptis.
\newblock K-theory of derivators revisited.
\newblock Preprint, available at \url{http://arxiv.org/abs/1402.1871}, 2014.

\bibitem[Nee92]{Neeman92b}
Amnon Neeman.
\newblock The connection between the {$K$}-theory localization theorem of
  {T}homason, {T}robaugh and {Y}ao and the smashing subcategories of
  {B}ousfield and {R}avenel.
\newblock {\em Ann. Sci. \'Ecole Norm. Sup. (4)}, 25(5):547--566, 1992.

\bibitem[Nee01]{Neeman01}
Amnon Neeman.
\newblock {\em Triangulated categories}, volume 148 of {\em Annals of
  Mathematics Studies}.
\newblock Princeton University Press, 2001.

\bibitem[TT90]{ThomasonTrobaugh90}
R.~W. Thomason and T.~Trobaugh.
\newblock Higher algebraic {$K$}-theory of schemes and of derived categories.
\newblock In {\em The Grothendieck Festschrift, Vol.\ III}, volume~88 of {\em
  Progr. Math.}, pages 247--435. Birkh\"auser, Boston, MA, 1990.

\bibitem[Ver96]{Verdier96}
Jean-Louis Verdier.
\newblock Des cat\'egories d\'eriv\'ees des cat\'egories ab\'eliennes.
\newblock {\em Ast\'erisque}, (239):xii+253 pp. (1997), 1996.
\newblock Edited and with a note by Georges Maltsiniotis.

\end{thebibliography}

\end{document}